\documentclass[english]{article}
\usepackage[T1]{fontenc}
\usepackage[latin9]{inputenc}
\usepackage{babel}
\usepackage{wrapfig}
\usepackage{amsmath}
\usepackage{amsthm}
\usepackage{amssymb}
\usepackage[numbers,sort&compress]{natbib}
\usepackage{microtype}
\usepackage[unicode=true,
 bookmarks=false,
 breaklinks=false,pdfborder={0 0 1},backref=section,colorlinks=false]
 {hyperref}

\makeatletter
\numberwithin{equation}{section}
\theoremstyle{plain}
\newtheorem{thm}{\protect\theoremname}[section]
\theoremstyle{plain}
\newtheorem{cor}[thm]{\protect\corollaryname}
\theoremstyle{plain}
\newtheorem{lem}[thm]{\protect\lemmaname}
\theoremstyle{definition}
\newtheorem{defn}[thm]{\protect\definitionname}

\usepackage[final]{neurips_2020}
\usepackage{color}
\usepackage{url}

\addto\captionsenglish{}
\usepackage[labelsep=endash, font=footnotesize, labelfont=bf,skip=5pt]{caption}

\usepackage{tikz}
\usetikzlibrary{arrows}
\tikzset{
    thick/.style=      {line width=2pt},
    ultra thick/.style={line width=2.5pt}
}

\makeatother

\providecommand{\corollaryname}{Corollary}
\providecommand{\definitionname}{Definition}
\providecommand{\lemmaname}{Lemma}
\providecommand{\theoremname}{Theorem}

\begin{document}
\title{On the Tightness of Semidefinite Relaxations for Certifying Robustness
to Adversarial Examples}
\author{Richard Y.~Zhang\\
Department of Electrical and Computer Engineering,\\
University of Illinois at Urbana-Champaign, \\
Urbana, 61801 IL, USA.\\
\texttt{ryz@illinois.edu}}
\maketitle
\begin{abstract}
The robustness of a neural network to adversarial examples can be
provably certified by solving a convex relaxation. If the relaxation
is loose, however, then the resulting certificate can be too conservative
to be practically useful. Recently, a less conservative robustness
certificate was proposed, based on a semidefinite programming (SDP)
relaxation of the ReLU activation function. In this paper, we describe
a geometric technique that determines whether this SDP certificate
is \emph{exact}, meaning whether it provides both a lower-bound on
the size of the smallest adversarial perturbation, as well as a globally
optimal perturbation that attains the lower-bound. Concretely, we
show, for a least-squares restriction of the usual adversarial attack
problem, that the SDP relaxation amounts to the nonconvex projection
of a point onto a hyperbola. The resulting SDP certificate is exact
if and only if the projection of the point lies on the major axis
of the hyperbola. Using this geometric technique, we prove that the
certificate is exact over a single hidden layer under mild assumptions,
and explain why it is usually conservative for several hidden layers.
We experimentally confirm our theoretical insights using a general-purpose
interior-point method and a custom rank-2 Burer-Monteiro algorithm.
\end{abstract}

\section{Introduction}

\global\long\def\tr{\mathrm{tr}}%
\global\long\def\R{\mathbb{R}}%
\global\long\def\relu{\mathrm{ReLU}}%
\global\long\def\diag{\mathrm{diag}}%
\global\long\def\b#1{\mathbf{#1}}%
\global\long\def\o#1{\overline{#1}}%
\global\long\def\x{\mathbf{x}}%
\global\long\def\e{\mathbf{e}}%
\global\long\def\u{\mathbf{u}}%
\global\long\def\v{\mathbf{v}}%
\global\long\def\z{\mathbf{z}}%
\global\long\def\bvec{\mathbf{b}}%
\global\long\def\Wmat{\mathbf{W}}%
\global\long\def\X{\mathbf{X}}%
\global\long\def\G{\mathbf{G}}%
\global\long\def\ub{\mathrm{ub}}%
\global\long\def\lb{\mathrm{lb}}%
\global\long\def\rank{\mathrm{rank}}%
\global\long\def\sign{\mathrm{sign}}%
\global\long\def\sumsm{{\textstyle \sum}}%
\global\long\def\forall{\mbox{for all }}%
It is now well-known that neural networks are vulnerable to \emph{adversarial
examples}: imperceptibly small changes to the input that result in
large, possibly targeted change to the output~\citep{biggio2013evasion,Szegedy2014intriguing,carlini2017towards}.
Adversarial examples are particularly concerning for safety-critical
applications like self-driving cars and smart grids, because they
present a mechanism for erratic behavior and a vector for malicious
attacks. 

Methods for analyzing robustness to adversarial examples work by formulating
the problem of finding the \emph{smallest} perturbation needed to
result in an adversarial outcome. For example, this could be the smallest
change to an image of the digit ``3'' for a given model to misclassify
it as an ``8''. The \emph{size} of this smallest change serves as
a \emph{robustness margin}: the model is robust if even the smallest
adversarial change is still easily detectable. 

Computing the robustness margin is a nonconvex optimization problem.
In fact, methods that \emph{attack} a model work by locally solving
this optimization, usually using a variant of gradient descent~\citep{43405,kurakin2016adversarial,madry2017towards,carlini2017towards}.
A successful attack demonstrates vulnerability by explicitly stating
a small---but not necessarily the smallest---adversarial perturbation.
Of course, failed attacks do not prove robustness, as there is always
the risk of being defeated by stronger attacks in the future. Instead,
robustness can be \emph{certified} by proving lower-bounds on the
robustness margin~\citep{katz2017reluplex,ehlers2017formal,huang2017safety,wong2018provable,wong2018scaling,dvijotham2018dual,dvijotham2018training,mirman2018differentiable,singh2018fast,gowal2018effectiveness,NIPS2018_7742,weng2018towards}.
Training against a robustness certificate (as an adversary) in turn
produces models that are certifiably robust to adversarial examples~\citep{sinha2017certifiable,wong2018provable,raghunathan2018certified}. 

The most useful robustness certificates are \emph{exact}, meaning
that they also explicitly state an adversarial perturbation whose
size matches their lower-bound on the robustness margin, thereby proving
global optimality~\citep{katz2017reluplex,ehlers2017formal,huang2017safety}.
Unfortunately, the robustness certification problem is NP-hard in
general, so most existing methods for exact certification require
worst-case time that scales exponentially with respect to the number
of neurons. In contrast, \emph{conservative} certificates are more
scalable because the have polynomial worst-case time complexity~\citep{wong2018provable,wong2018scaling,dvijotham2018dual,dvijotham2018training,mirman2018differentiable,singh2018fast,gowal2018effectiveness,NIPS2018_7742,weng2018towards}.
Their usefulness is derived from their level of conservatism. The
issue is that a pessimistic assessement for a model that is ostensibly
robust can be attributed to either undue conservatism in the certificate,
or an undiscovered vulnerability in the model. Also, training against
an overly conservative certificate will result in an overly cautious
model that is too willing to sacrifice performance for perceived safety.

Recently, \citet{NIPS2018_8285} proposed a less conservative certificate
based on a \emph{semidefinite programming }(SDP) relaxation of the
rectified linear unit (ReLU) activation function. Their empirical
results found it to be significantly less conservative than competing
approaches, based on linear programming or propagating Lipschitz constants.
In other domains, ranging from integer programming~\citep{lovasz1991cones,lasserre2001explicit},
polynomial optimization~\citep{lasserre2001global,parrilo2003semidefinite},
matrix completion~\citep{candes2009exact,candes2010power}, to matrix
sensing~\citep{recht2010guaranteed}, the SDP relaxation is often
\emph{tight}, in the sense that it is formally equivalent to the original
combinatorially hard problem. Within our context, tightness corresponds
to exactness in the robustness certificate. Hence, the SDP relaxation
is a good candidate for exact certification in polynomial time, possibly
over some restricted class of models or datasets.

This paper aims to understand when the SDP relaxation of the ReLU
becomes tight, with the goal of characterizing conditions for exact
robustness certification. Our main contribution is a geometric technique
for analyzing tightness, based on splitting a least-squares restriction
of the adversarial attack problem into a sequence of projection problems.
The final problem projects a point onto a nonconvex hyperboloid (i.e.
a high-dimensional hyperbola), and the SDP relaxation is tight if
and only if this projection lies on the major axis of the hyperboloid.
Using this geometric technique, we prove that the SDP certificate
is generally exact for a single hidden layer. The certificate is usually
conservative for several hidden layers; we use the same geometric
technique to offer an explanation for why this is the case.

\textbf{Notations. }Denote \emph{vectors} in boldface lower-case
$\x$, \emph{matrices} in boldface upper-case $\X$, and \emph{scalars}
in non-boldface $x,X$. The bracket denotes \emph{indexing} $\x[i]$
starting from 1, and also \emph{concatenation}, which is row-wise
via the comma $[a,b]$ and column-wise via the semicolon $[a;b]$.
The $i$-th canonical basis vector $\e_{i}$ satisfies $\e_{i}[i]=1$
and $\e_{i}[j]=0$ for all $j\ne i$. The usual \emph{inner product}
is $\langle\b a,\b b\rangle=\sumsm_{i}\b a[i]\b b[i]$, and the usual
rectified linear unit activation function is $\relu(\x)\equiv\max\{\x,0\}$. 

\section{Main results}

Let $\b f:\R^{n}\to\R^{m}$ be a feedforward ReLU neural network classifier
with $\ell$ hidden layers
\begin{equation}
\b f(\x_{0})=\x_{\ell}\text{ where }\quad\x_{k+1}=\relu(\Wmat_{k}\x_{k}+\bvec_{k})\quad\text{ for all }k\in\{0,1,\ldots,\ell-1\},\label{eq:nn}
\end{equation}
that takes an input $\hat{\x}\in\R^{n}$ (say, an image of a hand-written
single digit) labeled as belonging to the $i$-th of $m$ classes
(say, the 5-th of 10 possible classes of single digits), and outputs
a prediction vector $\b s=\Wmat_{\ell}\b f(\hat{\x})+\bvec_{\ell}\in\R^{m}$
whose $i$-th element is the largest, as in $\b s[i]>\b s[j]$ for
all $j\ne i$. Then, the problem of finding an adversarial example
$\x$ that is similar to $\hat{\x}$ but causes an incorrect $j$-th
class to be ranked over the $i$-th class can be posed 
\begin{equation}
d_{j}^{\star}\quad=\quad\min_{\x\in\R^{n}}\quad\|\x-\hat{\x}\|\quad\text{subject to}\quad\text{(\ref{eq:nn})},\quad\langle\b w,\b f(\x)\rangle+b\le0,\tag{A}\label{eq:probA}
\end{equation}
where $\b w=\Wmat_{\ell}^{T}(\e_{i}-\e_{j})$ and $b=\b b_{\ell}^{T}(\e_{i}-\e_{j}).$
In turn, the adversarial example $\x^{\star}$ most similar to $\hat{\x}$
over \emph{all} incorrect classes gives a \emph{robustness margin}
$d^{\star}=\min_{j\ne i}d_{j}$ for the neural network. 

In practice, the SDP relaxation for problem (\ref{eq:probA}) is often
loose. To understand the underlying mechanism, we study a slight modification
that we call its \emph{least-squares restriction}
\begin{equation}
L^{\star}\quad=\quad\min_{\x\in\R^{n}}\quad\|\x-\hat{\x}\|\quad\text{subject to}\quad\text{(\ref{eq:nn})},\qquad\|\b f(\x)-\hat{\z}\|\le\rho,\tag{B}\label{eq:probB}
\end{equation}
where $\hat{\z}\in\R^{m}$ is the targeted output, and $\rho>0$ is
a radius parameter. Problem (\ref{eq:probA}) is equivalent to problem
(\ref{eq:probB}) taken at the limit $\rho\to\infty$, because a half-space
is just an infinite-sized ball
\begin{equation}
\|\z+\b w\left(b/\|\b w\|^{2}+\rho/\|\b w\|\right)\|^{2}\le\rho^{2}\quad\iff\quad{\textstyle \frac{\|\b w\|}{2\rho}}\|\z+\b w\left(b/\|\b w\|^{2}\right)\|^{2}+[\langle\b w,\z\rangle+b]\le0\label{eq:ball-halfspace}
\end{equation}
with a center $\hat{\z}=-\b w\left(b/\|\b w\|^{2}+\rho/\|\b w\|\right)$
that tends to infinity alongside the ball radius $\rho$. The SDP
relaxation for problem (\ref{eq:probB}) is often tight for finite
values of the radius $\rho$. The resulting solution $\x$ is a \emph{strictly}
feasible (but suboptimal) attack for problem (\ref{eq:probA}) that
causes misclassification $\langle\b w,\b f(\x)\rangle+b<0$. The corresponding
optimal value $L^{\star}$ gives an upper-bound $d_{\ub}\equiv L^{\star}\ge d^{\star}$
that converges to an equality as $\rho\to\infty$. (See Appendix~\ref{sec:general}
for details.)

In Section~\ref{sec:oneneuron}, we completely characterize the SDP
relaxation for problem (\ref{eq:probB}) over a single hidden neuron,
by appealing to the underlying geometry of the relaxation. In Section~\ref{sec:onelayer},
we extend these insights to partially characterize the case of a single
hidden layer. 
\begin{thm}[One hidden neuron]
\label{thm:neuron}Consider the one-neuron version of problem (\ref{eq:probB}),
explicitly written
\begin{equation}
L^{\star}\quad=\quad\min_{x}\quad|x-\hat{x}|\quad\text{subject to}\quad|\relu(x)-\hat{z}|\le\rho.\label{eq:probB1}
\end{equation}
The SDP relaxation of (\ref{eq:probB1}) yields a tight lower-bound
$L_{\lb}=L^{\star}$ and a globally optimal $x^{\star}$ satisfying
$|x^{\star}-\hat{x}|=L_{\lb}$ if and only if one of the two conditions
hold: (i) $\rho\ge|\hat{z}|$; or (ii) $\rho<\hat{z}/(1-\min\{0,\hat{x}/\hat{z}\})$. 
\end{thm}

\begin{thm}[One hidden layer]
\label{thm:layer}Consider the one-layer version of problem (\ref{eq:probB}),
explicitly written
\begin{equation}
L^{\star}\quad=\quad\min_{\x\in\R^{n}}\quad\|\x-\hat{\x}\|\quad\text{ s.t. }\quad\|\relu(\Wmat\x)-\hat{\z}\|\le\rho\label{eq:probB2}
\end{equation}
The SDP relaxation of (\ref{eq:probB1}) yields a tight lower-bound
$L_{\lb}=L^{\star}$ and a globally optimal $\x^{\star}$ satisfying
$\|\x^{\star}-\hat{\x}\|=L_{\lb}$ if one of the two conditions hold:
(i) $\rho\ge\|\relu(\Wmat\hat{\x})-\hat{\z}\|$; or (ii) $\rho<\hat{z}_{\min}/2(1+\kappa)$
and $\|\b W\hat{\x}-\hat{\b z}\|_{\infty}<\hat{z}_{\min}^{2}/(2\rho\kappa)$
where $\hat{z}_{\min}=\min_{i}\hat{z}_{i}$ and $\kappa=\|\b W\|^{2}\|(\b W\b W^{T})^{-1}\|_{\infty}$.
\end{thm}

The lack of a weight term in (\ref{eq:probB1}) and a bias term in
(\ref{eq:probB1}) and (\ref{eq:probB2}) is without loss of generality,
as these can always be accommodated by shifting and scaling $\x$
and $\hat{\x}$. Intuitively, Theorem~\ref{thm:neuron} and Theorem~\ref{thm:layer}
say that the SDP relaxation tends to be tight if the output target
$\hat{\z}$ is \emph{feasible}, meaning that there exists some choice
of $\b u$ such that $\hat{\z}=\b f(\b u)$. (The condition $\rho<\hat{z}_{\min}/2(1+\kappa)$
is sufficient for feasibility.) Conversely, the SDP relaxation tends
to be loose if the radius $\rho>0$ lies within an intermediate band
of ``bad'' values. For example, over a single neuron with a feasible
$\hat{z}=1$, the relaxation is loose if and only if $\hat{x}\le0$
and $1/(1+|\hat{x}|)\le\rho<1$. These two general trends are experimentally
verified in Section~\ref{sec:Numerical-experiments}.

In the case of multiple layers, the SDP relaxation is usually loose,
with a notable exception being the trivial case with $L^{\star}=0$. 
\begin{cor}[Multiple layers]
\label{cor:multilayer}If $\rho\ge\|\b f(\hat{\x})-\hat{\z}\|$,
then the SDP relaxation of problem (\ref{eq:probB}) yields the tight
lower-bound $L_{\lb}=L^{\star}=0$ and the globally optimal $\x^{\star}=\hat{\x}$
satisfying $\|\x^{\star}-\hat{\x}\|=0$.
\end{cor}

The proof is given in Appendix~\ref{sec:general}. In Section~\ref{sec:multiple},
we explain the looseness of the relaxation for multiple layers using
the geometric insight developed for the single layer. As mentioned
above, the general looseness of the SDP relaxation for problem (\ref{eq:probB})
then immediately implies the general looseness for problem (\ref{eq:probA}).

\section{Related work }

\textbf{Adversarial attacks, robustness certificates, and certifiably
robust models.} Adversarial examples are usually found by using projected
gradient descent to solve problem (\ref{eq:probA}) with its objective
and constraint swapped~\citep{43405,kurakin2016adversarial,madry2017towards,carlini2017towards}.
Training a model against these empirical attacks generally yield very
resilient models~\citep{43405,kurakin2016adversarial,madry2017towards}.
It is possible to certify robustness exactly despite the NP-hardness
of the problem~\citep{katz2017reluplex,ehlers2017formal,huang2017safety,tjeng2017evaluating}.
Nevertheless, conservative certificates show greater promise for scalability
because they are polynomial-time algorithms. From the perspective
of tightness, the next most promising techniques after the SDP relaxation
are relaxations based on linear programming (LP)~\citep{wong2018provable,wong2018scaling,dvijotham2018dual,dvijotham2018training},
though techniques based on propagating bounds and/or Lipschitz constants
tend to be much faster in practice~\citep{mirman2018differentiable,singh2018fast,gowal2018effectiveness,weng2018towards,NIPS2018_7742}.
Aside from training a model against a robustness certificate~\citep{sinha2017certifiable,wong2018provable,raghunathan2018certified},
certifiably robust models can also be constructed by randomized smoothing~\citep{lecuyer2019certified,cohen2019certified,salman2019provably}.

\textbf{Tightness of SDP relaxations.} The geometric techniques used
in our analysis are grounded in the classic paper of \citet{goemans1995improved}
(see also~\citep{frieze1997improved,nesterov1997quality,so2007approximating}),
but our focuses are different: they prove general bounds valid over
entire classes of SDP relaxations, whereas we identify specific SDP
relaxations that are exactly tight. In the sense of tight relaxations,
our results are reminescent of the guarantees by \citet{candes2009exact,candes2010power}
(see also~\citep{candes2010matrix,JMLR:v12:recht11a}) on the matrix
completion problem, but our approaches are very different: their arguments
are based on using the dual to imply tightness in the primal, whereas
our proof analyzes the primal directly. 

After this paper was submitted, we became aware of two parallel work~\citep{fazlyab2019safety,dvijotham2020efficient}
that also study the tightness of SDP relaxations for robustness to
adversarial examples. The first, due to \citet{fazlyab2019safety},
uses similar techniques like the S-procedure to study a different
SDP relaxation constructed from robust control techniques. The second,
due to \citet{dvijotham2020efficient}, studies the same SDP relaxation
within the context of a different attack problem, namely the version
of Problem~(\ref{eq:probA}) with the objective replaced by the infinity
norm distance $\|\x-\hat{\x}\|_{\infty}$.

\textbf{Efficient algorithms for SDPs.} While SDPs are computationally
expensive to solve using off-the-shelf algorithms, efficient formulation-specific
solvers were eventually developed once their use case became sufficiently
justified. In fact, most state-of-the-art algorithms for phase retrieval~\citep{netrapalli2013phase,chen2015solving,bhojanapalli2016global,goldstein2018phasemax}
and collaborative filtering~\citep{cai2010singular,keshavan2010matrixa,keshavan2010matrixb,sun2016guaranteed,ge2016matrix}
can be viewed as highly optimized algorithms to solve an underlying
SDP relaxation. Our rank-2 Burer-Monteiro algorithm in Section~\ref{sec:Numerical-experiments}
is inspired by~\citet{burer2002rank}. It takes strides towards an
efficient algorithm, but the primary focus of this paper is to understand
the use case for robustness certification. 

\section{\label{sec:sdp_relax}Preliminary: Geometry of the SDP relaxation}

The SDP relaxation of \citet{NIPS2018_8285} is based on the observation
that the rectified linear unit (ReLU) activation function $z=\relu(x)\equiv\max\{0,x\}$
is equivalent to the inequalities $z\ge0,$ $z\ge x,$ and $z(z-x)\le0$.
Viewing these as quadratics, we apply a standard technique (see \citet{Shor19871}
and also~\citep{goemans1995improved,lasserre2001global,parrilo2003semidefinite})
to rewrite them as linear inequalities over a positive semidefinite
matrix variable, 
\begin{equation}
z\ge0,\quad z\ge x,\quad Z\le Y,\quad\G=\begin{bmatrix}1 & x & z\\
x & X & Y\\
z & Y & Z
\end{bmatrix}\succeq0,\quad\rank(\G)=1.\label{eq:relaxmat}
\end{equation}
In essence, the reformulation collects the inherent nonconvexity of
$\relu(\cdot)$ into the constraint $\rank(\G)=1$, which can then
be deleted to yield a convex relaxation. If the relaxation has a unique
solution $\G^{\star}$ satisfying $\rank(\G^{\star})=1$, then we
say that it is \emph{tight}.\footnote{If a rank-1 solution exists but is nonunique, then we do not consider
the SDP relaxation tight because the rank-1 solution cannot usually
be found in polynomial time. Indeed, an interior-point method converges
onto a maximum rank solution, but this can be rank-1 only if it is
unique.} In this case, the globally optimal solution $x^{\star},z^{\star}$
to the original nonconvex problem can be found by solving the SDP
relaxation in polynomial time and factorizing the solution $\G^{\star}=\b g\b g^{T}$
where $\b g=[1;x^{\star};z^{\star}]^{T}$. 

It is helpful to view $\G$ as the \emph{Gram matrix} associated with
the vectors $\e,\x,\z\in\R^{p}$ in an ambient $p$-dimensional space,
where $p$ is the order of $\G$ (here $p=3$). The individual elements
of $\G$ correspond to the inner products terms associated with $\e,\x,\z$,
as in
\begin{equation}
\langle\e,\z\rangle\ge\max\{0,\langle\e,\x\rangle\},\quad\|\z\|^{2}\le\langle\z,\x\rangle,\quad\|\e\|^{2}=1,\quad\G=\begin{bmatrix}\langle\e,\e\rangle & \langle\e,\x\rangle & \langle\e,\z\rangle\\
\langle\e,\x\rangle & \langle\x,\x\rangle & \langle\x,\z\rangle\\
\langle\e,\z\rangle & \langle\x,\z\rangle & \langle\z,\z\rangle
\end{bmatrix},\label{eq:relaxnonconv}
\end{equation}
and $\rank(\G)=1$ corresponds to \emph{collinearity} between $\x$,
$\z$, and $\e$, as in $\|\e\|\|\x\|=|\langle\e,\x\rangle|$ and
$\|\e\|\|\z\|=|\langle\e,\z\rangle|$. From the Gram matrix perspective,
the SDP relaxation works by allowing the underlying vectors $\x$,
$\z$, and $\e$ to take on arbitrary directions; the relaxation is
tight if and only if \emph{all} possible solutions $\e^{\star},\x^{\star},\z^{\star}$
are collinear. 

\vspace{-1em} ~ \begin{wrapfigure}{r}{0.25\columnwidth}%
\vspace{-2em}

\centering
\begin{tikzpicture}[scale=0.6]
\fill[red!10] (1.5,-1) circle (1.8);
\fill[red!10] (-0.5,0) rectangle (4,1);
\begin{scope}
  \clip (0,0) rectangle (4,1);
  \fill[blue!50] (1.5,-1) circle (1.8);
\end{scope}
\draw[thick, red] (-0.5,0) -- (4,0); 

\coordinate (xvec) at (3,0);
\coordinate (xhalf) at (1.5,-1);
\coordinate (origin) at (0,-2);

\draw[dashed] (0,1) -- (0,-3); 
\draw[ultra thick,-latex] (origin) -- (xvec) node[above right]{\Large $\mathbf{x}$}; 
\draw[ultra thick,-latex] (origin) -- ++(0,1) node[right] {\Large $\mathbf{e}$}; 
\draw[thick, red] (xhalf) circle (1.8);
\fill[black] (xhalf) node[below right] {\Large $\mathbf{x}/2$} circle (5pt);
\end{tikzpicture}

\caption{\footnotesize \label{fig:feas_region}The ReLU constraints (\ref{eq:relaxnonconv})
describe a spherical cap. }

\vspace{-1em}
\end{wrapfigure}
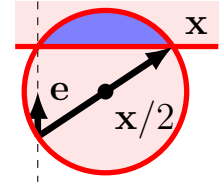%
 Figure~\ref{fig:feas_region} shows the underlying geometry the
ReLU constraints (\ref{eq:relaxnonconv}) as noted by \citet{NIPS2018_8285}.
Take $\z$ as the variable and fix $\e,\x$. Since $\e$ is a unit
vector, we may view $\langle\e,\x\rangle$ and $\langle\e,\z\rangle$
as the ``$\e$-axis coordinates'' for the vectors $\x$ and $\z$.
The constraint $\langle\e,\z\rangle\ge\max\{0,\langle\e,\x\rangle\}$
is then a \emph{half-space} that restricts the ``$\e$-coordinate''
of $\z$ to be nonnegative and greater than that of $\x$. The constraint
$\langle\z,\z-\x\rangle\le0$ is rewritten as $\|\z-\x/2\|^{2}\le\|\x/2\|^{2}$
by completing the square; this is a \emph{sphere} that restricts $\z$
to lie within a distance of $\|\x/2\|$ from the center $\x/2$. Combined,
the ReLU constraints (\ref{eq:relaxnonconv}) constrain $\z$ to lie
within a \emph{spherical cap}---a portion of a sphere cut off by
a plane.

\section{\label{sec:oneneuron}Tightness for one hidden neuron}

Now, consider the SDP relaxation of the one-neuron problem (\ref{eq:probB1}),
explicitly written as
\begin{equation}
L_{\lb}^{2}=\min_{\G}\quad X-2x\hat{x}+\hat{x}^{2}\quad\text{s.t.}\quad\begin{array}{c}
z\ge\max\{0,x\},\;Z\le Y,\\
Z-2z\hat{z}+\hat{z}^{2}\le\rho^{2},
\end{array}\quad\G=\begin{bmatrix}1 & x & z\\
x & X & Y\\
z & Y & Z
\end{bmatrix}\succeq0.\label{eq:relaxB1}
\end{equation}
Viewing the matrix variable $\G\succeq0$ as the Gram matrix associated
with the vectors $\e,\x,\z\in\R^{p}$ where $p=3$ rewrites (\ref{eq:relaxB1})
as the following 
\begin{align}
L_{\lb}=\min_{\x,\z,\e\in\R^{p}} & \|\x-\hat{x}\,\e\|\quad\text{s.t. }\langle\z,\e\rangle\ge\max\{\langle\x,\e\rangle,0\},\;\|\z\|^{2}\le\langle\z,\x\rangle,\;\|\z-\hat{z}\,\e\|\le\rho.\label{eq:ncvxB1}
\end{align}
The SDP relaxation (\ref{eq:relaxB1}) has a unique rank-1 solution
if and only if its nonconvex vector interpretation (\ref{eq:ncvxB1})
has a unique solution that aligns with $\e$. The proof for the following
is given in Appendix~\ref{sec:collinear}.
\begin{lem}[Collinearity and rank-1]
\label{lem:collinear1}Fix $\e\in\R^{p}$. Then, problem (\ref{eq:ncvxB1})
has a unique solution $\x^{\star}$ satisfying $\|\x^{\star}\|=|\langle\x^{\star},\e\rangle|$
if and only if problem (\ref{eq:relaxB1}) has a unique solution $\G^{\star}$
satisfying $\rank(\G^{\star})=1$.
\end{lem}

We proceed to solve problem (\ref{eq:ncvxB1}) by rewriting it as
the composition of a convex projection over $\z$ with a nonconvex
projection over $\x$, as in:
\begin{align}
\phi(\x,\hat{z})\quad & =\quad\min_{\z\in\R^{p}}\quad\|\z-\hat{z}\e\|\quad\text{subject to }\quad\langle\e,\z\rangle\ge\max\{\langle\e,\x\rangle,0\},\quad\|\z\|^{2}\le\langle\z,\x\rangle,\label{eq:proj_z}\\
L_{\lb}\quad & =\quad\min_{\x\in\R^{p}}\quad\|\x-\hat{x}\e\|\quad\text{subject to }\quad\phi(\x,\hat{z})\le\rho.\label{eq:proj_x}
\end{align}
\vspace{-1em} ~ 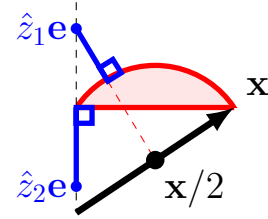
\begin{wrapfigure}{r}{0.25\columnwidth}%
\vspace{-1em}

\mbox{\centering
\begin{tikzpicture}[scale=0.7]
\begin{scope}
  \clip (0,0) rectangle (4,2);
  \filldraw[thick, color=red, fill=red!10] (1.5,-1) circle (1.8);
\end{scope}
\draw[thick, red] (0,0) -- (3,0); 

\coordinate (xvec) at (3,0);
\coordinate (xhalf) at (1.5,-1);
\coordinate (origin) at (0,-2);
\coordinate (z1) at (0, 1.5);
\coordinate (z2) at (0,-1.5);
\coordinate (proj2) at (0.6,0.5);

\draw[dashed] (0,2) -- (0,-2); 
\draw[ultra thick,-latex] (origin) -- (xvec) node[above right]{\Large $\mathbf{x}$}; 
\draw[thick, blue] (z2) -- (0,0) rectangle ++(0.3,-0.3); 
\draw[dashed, red] (proj2) -- (xhalf); 
\draw[thick, blue] (z1) -- (proj2); 
\draw[thick, blue, rotate around={32:(proj2)}] (proj2) rectangle ++(0.3,0.3);

\fill[black] (xhalf) node[below right] {\Large $\mathbf{x}/2$} circle (5pt);
\fill[blue] (z1) node[left] {\Large $\hat{z}_1\mathbf{e}$} circle (3pt);
\fill[blue] (z2) node[left] {\Large $\hat{z}_2\mathbf{e}$} circle (3pt);

\end{tikzpicture}}

\caption{\footnotesize \label{fig:proj_z}Problem (\ref{eq:proj_z}) is the
projection of a point onto a spherical cap.}
\end{wrapfigure}%
 Problem (\ref{eq:proj_z}) is clearly the projection of the point
$\hat{z}\e$ onto the spherical cap shown in Figure~\ref{fig:feas_region}.
In a remarkable symmetry, it turns out that problem (\ref{eq:proj_x})
is the projection of the point $\hat{x}\e$ onto a \emph{hyperboloidal
cap}---a portion of a high-dimensional hyperbola cut off by a plane---with
the optimal $\x^{\star}$ in problem (\ref{eq:ncvxB1}) being the
resulting projection. In turn, our goal of verifying collinearity
between $\x^{\star}$ and $\e$ amounts to checking whether $\hat{x}\e$
projects onto the major axis of the hyperboloid.

To turn this intuitive sketch into a rigorous proof, we begin by solving
the convex projection (\ref{eq:proj_z}) onto the spherical cap. Figure~\ref{fig:proj_z}
shows the corresponding geometry. There are two distinct scenarios:
\textbf{i)}~For $\hat{z}_{1}\e$ that is \emph{above} the spherical
cap, the projection must intersect the upper, round portion of the
spherical cap along the line from $\hat{z}_{1}\e$ to $\x/2$. This
yields a distance of $\phi(\x,\hat{z}_{1})=\|\hat{z}_{1}\e-\x/2\|-\|\x/2\|$.
\textbf{ii)~}For $\hat{z}_{2}\e$ that is \emph{below} the spherical
cap, the projection is simply the closest point directly above, at
a distance of $\phi(\x,\hat{z}_{2})=\max\{0,\langle\e,\x\rangle\}-\hat{z}_{2}$. 

\vspace{-1em} ~ \begin{wrapfigure}{r}{0.25\columnwidth}%
\vspace{-1em}

\centering
\begin{tikzpicture}[scale=0.75]
\begin{scope}[even odd rule]
    \clip (-2,-2.5) parabola bend (0,-1) (2,-2.5) rectangle (-2,1);
    \fill[red!10] (2,-3) rectangle (-2,1);
\end{scope}
\draw[thick, red] (-2,1) -- (2,1);
\draw[thick, red] (-2,-2.5) parabola bend (0,-1) (2,-2.5);

\draw[dashed] (0,2.5) -- (0,-3.5); 
\draw[red,dashed] (-2,0) -- (2,0); 

\coordinate (origin) at (0,-2);
\coordinate (z1) at (0, 0);
\coordinate (z2) at (0, 2);
\coordinate (x1) at (0, 2);
\coordinate (x2) at (0, -2);
\coordinate (x3) at (0, -3);

\fill[black] (z1) node[above left] {\Large $\hat{z}\mathbf{e}$} circle (3pt);
\fill[blue] (x1) node[left] {\Large $\hat{x}_1\mathbf{e}$} circle (3pt);
\fill[blue] (x2) node[left] {\Large $\hat{x}_2\mathbf{e}$} circle (3pt);
\fill[blue] (x3) node[left] {\Large $\hat{x}_3\mathbf{e}$} circle (3pt);

\draw[red] (0.6, 0.5) node[right] {\Large $\rho$};
\draw[red] (0.6, -0.5) node[right] {\Large $\rho$};
\draw[red, thick, latex-latex] (0.5, 0) -- (0.5, 1);
\draw[red, thick, latex-latex] (0.5, 0) -- (0.5, -1);

\draw[thick, blue] (x1) -- (0,1) rectangle ++(0.3,0.3);
\draw[thick, blue] (x2) -- (0,-1) rectangle ++(0.3,-0.3);
\coordinate (proj3) at (1.6, -2);
\draw[thick, blue] (x3) -- (proj3);
\draw[thick, blue, rotate around={35:(proj3)}] (proj3) rectangle ++(-0.3,-0.3);
\end{tikzpicture}

\caption{\footnotesize \label{fig:proj_x}Problem (\ref{eq:proj_x}) is the
projection of a point onto a hyperbolidal cap.}
\end{wrapfigure}%
 It turns out that the conditional statements are unnecessary; the
distance $\phi(\x,\hat{z})$ simply takes on the larger of the two
values derived above. In Appendix~\ref{sec:reluproj}, we prove this
claim algebraically, thereby establishing the following.
\begin{lem}[Projection onto spherical cap]
\label{lem:reluproj}The function $\phi:\R^{p}\times\R\to\R$ defined
in (\ref{eq:proj_z}) satisfies
\begin{equation}
\phi(\x,\hat{z})=\max\{\max\{0,\langle\e,\x\rangle\}-\hat{z},\quad\|\hat{z}\e-\x/2\|-\|\x/2\|\}.\label{eq:hypersec}
\end{equation}
\end{lem}

Taking $\x$ as the variable, we see from Lemma~\ref{lem:reluproj}
that each level set $\phi(\x,\hat{z})=\rho$ is either: 1) a \emph{hyperplane}
normal to $\e$ at the intercept $\hat{z}+\rho$; or 2) a two-sheet
\emph{hyperboloid} centered at $\hat{z}\e$, with semi-major axis
$\rho$ and focal distance $|\hat{z}|$ in the direction of $\e$.
Hence, the sublevel set $\phi(\x,\hat{z})\le\rho$ is a hyperboloidal
cap as claimed. 

We proceed to solve the nonconvex projection (\ref{eq:proj_x}) onto
the hyperboloidal cap. This shape degenerates into a half-space if
the semi-major axis $\rho$ is longer than the focal distance, as
in $\rho\ge|\hat{z}|$, and becomes empty altogether with a negative
center, as in $\hat{z}<-\rho$. Figure~\ref{fig:proj_x} shows the
geometry of projecting onto a \emph{nondegenerate} hyperboloidal cap
with $\hat{z}>\rho$. There are three distinct scenarios:\textbf{
i)}~For $\hat{x}_{1}\e$ that is either \emph{above} or \emph{interior}
to the hyperboloidal cap, the projection is either the closest point
directly below or the point itself, as in $\x^{\star}=\min\{\hat{z}+\rho,\hat{x}_{1}\}\e$;
\textbf{ii)~}For $\hat{x}_{2}\e$ that is \emph{below} and \emph{sufficiently
close} to the hyperboloidal cap, the projection lies at the top of
the hyperbolid sheet at $\x^{\star}=(\hat{z}-\rho)\e$; \textbf{iii)~}For
$\hat{x}_{3}\e$ that is \emph{below} and \emph{far away} from the
hyperboloidal cap, the projection lies somewhere along the side of
the hyperboloid.

Evidently, the first two scenarios correspond to choices of $\x^{\star}$
that are collinear to $\e$, while the third scenario does not. To
resolve the boundary between the second and third scenarios, we solve
the projection onto a hyperbolidal cap in closed-form. 
\begin{lem}[Projection onto nondegenerate hyperboloidal cap]
\label{lem:hypproj1}Given $\e\in\R^{p},$ $\hat{x}\in\R,$ and $\hat{z}>\rho>0$,
define $\x^{\star}$ as the solution to the following projection 
\[
\min_{\x\in\R^{p}}\quad\|\x-\hat{x}\e\|^{2}\quad\text{s.t.}\quad\langle\e,\x\rangle-\hat{z}\le\rho,\quad\|\hat{z}\e-\x/2\|-\|\x/2\|\le\rho.
\]
Then, $\x^{\star}$ is unique and satisfies $\|\x^{\star}\|=|\langle\e,\x^{\star}\rangle|$
if and only if $(\hat{z}-\hat{x})<\hat{z}^{2}/\rho.$
\end{lem}

We defer the proof of Lemma~\ref{lem:hypproj1} to Appendix~\ref{sec:hypproj1},
but note that the main idea is to use the S-lemma (see e.g.~\citep[p.~655]{boyd2004convex}
or~\citep{polik2007survey}) to solve the minimization of a quadratic
(the distance) subject to a quadratic constraint (the nondegenerate
hyperboloid). Resolving the degenerate cases and applying Lemma~\ref{lem:hypproj1}
to the nondegenerate case yields a proof of our main result.
\begin{proof}[Proof of Theorem~\ref{thm:neuron}]
If $\hat{z}<-\rho$, then the hyperbolidal cap $\phi(\x,\hat{z})\le\rho$
is empty as $\rho<-\hat{z}\le\phi(\x,\hat{z})$. In this case, problem
(\ref{eq:proj_x}) is infeasible. If $|\hat{z}|\le\rho$, then the
hyperbolidal cap $\phi(\x,\hat{z})\le\rho$ degenerates into a half-space
$\langle\e,\x\rangle\le\hat{z}+\rho$, because $\|\hat{z}\e-\x/2\|-\|\x/2\|\le\|\hat{z}\e\|+\|\x/2\|-\|\x/2\|=|\hat{z}|\le\rho$.
In this case, the projection $\x^{\star}=\min\{\hat{z}+\rho,\hat{x}\}\e$
is clearly collinear to $\e$, so $|\hat{z}|\le\rho$ is the first
condition of Theorem~\ref{thm:neuron}. Finally, if $\hat{z}>\rho$,
Lemma~\ref{lem:hypproj1} says that $\x^{\star}$ is collinear with
$\e$ whenever $(\hat{z}-\hat{x})<\hat{z}^{2}/\rho$, which is rewritten
$\hat{x}>\hat{z}(1-\hat{z}/\rho)$. Under $\hat{z}>\rho$, this is
equivalent to $\rho<\hat{z}/(1-\hat{x}/\hat{z})$. Finally, taking
the intersection of these two constraints yields $\rho<\hat{z}/(1-\min\{0,\hat{x}/\hat{z}\})$,
which is the second condition of Theorem~\ref{thm:neuron}.
\end{proof}

\section{\label{sec:onelayer}Tightness for one layer}

Our analysis of the one-hidden-neuron case extends to the one-hidden-layer
case without significant modification. Here, the semidefinite relaxation
reads
\begin{align}
L_{\lb}^{2}=\quad\min_{\G}\quad & \tr(\b X)-2\langle\x,\hat{\x}\rangle+\|\hat{\x}\|^{2}\label{eq:relaxB2}\\
\text{s.t.}\quad & \begin{array}{c}
\z\ge\max\{0,\Wmat\x\},\;\diag(\Wmat\b Z)\le\diag(\Wmat\b Y),\\
\tr(\b Z)-2\langle\z,\hat{\z}\rangle+\|\hat{\z}\|^{2}\le\rho^{2},
\end{array}\;\G=\begin{bmatrix}1 & \x & \z\\
\x & \b X & \b Y\\
\z & \b Y^{T} & \b Z
\end{bmatrix}\succeq0.\nonumber 
\end{align}
Viewing the matrix variable $\G\succeq0$ in the corresponding SDP
relaxation (\ref{eq:relaxB2}) as the Gram matrix associated a set
of length-$p$ vectors (where $p=m+n+1$ is the order of the matrix
$\G$) yields the following\footnote{To avoid visual clutter we will abbreviate $\sum_{j=1}^{n}x_{j}$
and ``for all $i\in\{1,2,\ldots,n\}$'' as $\sum_{j}x_{j}$ and
``for all $i$'' whenever the ranges of indices are clear from context. }
\begin{align}
L_{\lb}^{2}=\quad\min_{\x_{j},\z_{i}\in\R^{p}}\quad & \sumsm_{j}\|\x_{j}-\hat{x}_{j}\,\e\|^{2}\label{eq:ncvxB2}\\
\text{s.t.}\quad & \begin{array}{c}
\langle\e,\z_{i}\rangle\ge\max\left\{ 0,\langle\e,{\textstyle \sum_{j}}W_{i,j}\x_{j}\rangle\right\} ,\\
\|\z_{i}\|^{2}\le\langle\z_{i},{\textstyle \sum_{j}}W_{i,j}\x_{j}\rangle,
\end{array}\quad\sumsm_{i}\|\z_{i}-\hat{z}_{i}\,\e\|^{2}\le\rho^{2}\;\forall i,\nonumber 
\end{align}
with indices $i\in\{1,2,\ldots,m\}$ and $j\in\{1,2,\ldots,n\}$.
We will derive conditions for the SDP relaxation (\ref{eq:relaxB2})
to have a unique, rank-1 solution by fixing $\e$ in problem (\ref{eq:ncvxB2})
and verifying that every optimal $\x_{j}^{\star}$ is collinear with
$\e$ for all $j$. The proof for the following is given in Appendix~\ref{sec:collinear}.
\begin{lem}[Collinearity and rank-1]
\label{lem:collinear2}Fix $\e\in\R^{p}$. Then, problem (\ref{eq:ncvxB2})
has a unique solution $\x_{1}^{\star},\x_{2}^{\star},\ldots,\x_{n}^{\star}$
satisfying $\|\x_{j}^{\star}\|=|\langle\x_{j}^{\star},\e\rangle|$
if and only if problem (\ref{eq:relaxB2}) has a unique solution $\G^{\star}$
satisfying $\rank(\G^{\star})=1$.
\end{lem}

Problem (\ref{eq:ncvxB2}) can be rewritten as the composition of
a series of projections over $\z_{i}$, followed by a sequence of
nonconvex projections over $\x_{j}$, as in 
\begin{equation}
L_{\lb}^{2}=\min_{\x_{j}\in\R^{p},a_{i}\ge0}\quad\sumsm_{j}\|\x_{j}-\hat{x}_{j}\,\e\|^{2}\quad\text{s.t.}\quad\phi(\sumsm_{j}W_{i,j}\x_{j},\hat{z}_{i})\le\rho_{i}\;\forall i,\quad\sumsm_{i}\rho_{i}^{2}\le\rho^{2},\label{eq:proj_xj}
\end{equation}
where $\phi$ was previously defined in the one-neuron convex projection
(\ref{eq:proj_z}). Whereas in the one-neuron case we are projecting
a single point onto a single hyperboloidal cap, the one-layer case
requires us to project $n$ points onto the intersection of $n$ hyperboloidal
caps. This has a closed-form solution only when all the hyperboloids
are nondegenerate.
\begin{lem}[Projection onto several hyperboloidal caps]
\label{lem:hypproj2}Given $\Wmat=[W_{i,j}]\in\R^{m\times n},$ $\hat{\x}=[\hat{x}_{j}]\in\R^{n},$
$\hat{\z}=[\hat{z}_{i}]\in\R^{m},$ $\e\in\R^{p},$ and $\rho_{i}$
satisfying $\hat{z}_{i}>\rho>0$, define $\x_{j}^{\star}$ as the
solution to the following projection 
\[
\min_{\x_{j}\in\R^{p}}\quad\sumsm_{j}\|\x_{j}-\hat{x}_{j}\e\|^{2}\quad\text{s.t.}\quad\begin{array}{r}
\langle\e,\sumsm_{j}W_{i,j}\x_{j}\rangle-\hat{z}_{i}\le\rho_{i}\;\forall i,\\
\|\hat{z}_{i}\e-\sumsm_{j}W_{i,j}\x_{j}/2\|-\|\sumsm_{j}W_{i,j}\x_{j}/2\|\le\rho_{i}\;\forall i.
\end{array}
\]
If $\rho_{\max}\|\b W\|^{2}\|(\b W\b W^{T})^{-1}(\b W\hat{\x}-\hat{\b z})\|_{\infty}+\rho_{\max}^{2}(1+\|\b W\|^{2}\|(\b W\b W^{T})^{-1}\|_{\infty})<\hat{z}_{\min}^{2}$
holds with $\rho_{\max}=\max_{i}\rho_{i}$ and $\hat{z}_{\min}=\min_{i}\hat{z}_{i}$,
then $\x_{j}^{\star}$ is unique and satisfies $\|\x_{j}^{\star}\|=|\langle\e,\x_{j}^{\star}\rangle|$
for all $j$.
\end{lem}

We defer the proof of Lemma~\ref{lem:hypproj2} to Appendix~\ref{sec:hypproj2},
but note that the main idea is to use the \emph{lossy} S-lemma to
solve the minimization of one quadratic (the distance) over several
quadratic constraints (the hyperboloids). Theorem~\ref{thm:layer}
then follows immediately from Lemma~\ref{lem:hypproj2} and Corollary~\ref{cor:multilayer}.

\section{\label{sec:multiple}Looseness for multiple layers}

Unfortunately, the SDP relaxation is not usually tight for more than
a single layer. Let $\b f(x)=\relu(\relu(x))$ denote a two-layer
neural network with a single neuron per layer. The corresponding instance
of problem (\ref{eq:probB}) is essentially the same as problem (\ref{eq:probB1})
from Section~\ref{sec:oneneuron} for the one-layer one-neuron network,
because $\relu(\relu(x))=\relu(x)$ holds for all $x$. However, constructing
the SDP relaxation and taking the Gram matrix interpretation reveals
the following (with $p=4$)
\begin{align}
L_{\lb}=\min_{\x,\z,\e\in\R^{p}} & \|\x-\hat{x}\,\e\|\quad\text{s.t. }\begin{array}{c}
\langle\z,\e\rangle\ge\max\{\langle\x,\e\rangle,0\},\;\|\z\|^{2}\le\langle\z,\x\rangle,\\
\langle\z,\e\rangle-\hat{z}\le\rho,\;\|\hat{z}\e-\z/2\|-\|\z/2\|\le\rho,
\end{array}\label{eq:proj_x3}
\end{align}
which is \emph{almost} the same as problem (\ref{eq:ncvxB1}) from
Section~\ref{sec:oneneuron}, except that the convex ball constraint
$\|\hat{z}\e-\z\|\le\rho$ has been replaced by a nonconvex hyperboloid.
As we will see, it is this hyperbolic geometry that makes it harder
for the SDP relaxation to be tight. 

\vspace{-1em} ~ \begin{wrapfigure}{r}{0.25\columnwidth}%
\vspace{-1em}

\centering
\mbox{\begin{tikzpicture}[scale=0.75]
\begin{scope}[even odd rule]
    \clip (-2,-2) parabola bend (0,-1) (2,-2) rectangle (-2,1);
    \fill[blue!10] (2,-2) rectangle (-2,1);
\end{scope}
\draw[thick, blue] (-2,1) -- (2,1);
\draw[thick, blue] (-2,-2) parabola bend (0,-1) (2,-2);

\filldraw[color=red, fill=red!30, thick] (0,0) circle (1);

\coordinate (zopt) at (0, -1);
\coordinate (zcirc) at (-0.9, -0.5);
\coordinate (zhyp) at (0.9, -1.2);
\coordinate (xhat) at (0, -2.5);

\fill[black] (zopt) circle (4pt);
\draw[black] (zopt)++(0,0.1) node[above]{\Large $\mathbf{u}$};
\fill[red] (zcirc) circle (4pt);
\draw[red] (zcirc)++(-0.1,0) node[left]{\Large $\mathbf{z}_c$};
\fill[blue] (zhyp) circle (4pt);
\draw[blue] (zhyp)++(0.1,0.1) node[right]{\Large $\mathbf{z}_h$};

\draw[ultra thick, black, ->] (xhat) -- (zopt);
\draw[ultra thick, red, ->] (xhat) -- (zcirc);
\draw[ultra thick, blue, ->] (xhat) -- (zhyp);
\end{tikzpicture}}

\caption{\footnotesize \label{fig:two_neuron} Any point $\protect\z_{c}$
on the circle clearly satisfies $\|\protect\z_{c}\|>\|\protect\u\|$,
but a point $\protect\z_{h}$ on the hyperbola may have $\|\protect\z_{h}\|\approx\|\protect\u\|$.}
\end{wrapfigure}
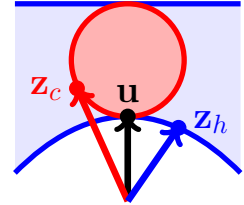%
 Denote $x^{\star}$ as the solution to both instances of problem
(\ref{eq:probB}). The point $\u=x^{\star}\e$ must be the unique
solution to (\ref{eq:proj_x3}) and (\ref{eq:ncvxB1}) in order for
their respective SDP relaxations to be tight. Now, suppose that $\hat{x}<0$
and $\hat{z}>\rho>0$, so that both instances of problem (\ref{eq:probB})
have $x^{\star}=\hat{z}-\rho>0$. Both (\ref{eq:ncvxB1}) and (\ref{eq:proj_x3})
are convex over $\x$; fixing $\z$ and optimizing over $\x$ in each
case yields $\|\x^{\star}-\hat{x}\e\|=\|\z\|-\hat{x}\cos\theta$ where
$\cos\theta=\langle\e,\z\rangle/\|\z\|$. In order for $\u$ to be
the unique solution, we need $\|\z\|-\hat{x}\cos\theta$ to be globally
minimized at $\z=\u$. As shown in Figure~\ref{fig:two_neuron},
$\|\z\|$ is clearly minimized at $\z^{\star}=\u$ over the ball constraint
$\|\hat{z}\e-\z\|\le\rho$, but the same is not obviously true for
the hyperbolid $\|\hat{z}\e-\z/2\|-\|\z/2\|\le\rho$. Some detailed
computation readily confirm the geometric intuition that $\u$ is
generally a local minimum over the circle, but not over the hyperbola.

\section{\label{sec:Numerical-experiments}Numerical experiments}

\textbf{Dataset and setup. }We use the MNIST dataset of $28\times28$
images of handwritten digits, consisting of 60,000 training images
and 10,000 testing images. We remove and set aside the final 1,000
images from the training set as the verification set. All our experiments
are performed on an Intel Xeon E3-1230 CPU (4-core, 8-thread, 3.2-3.6
GHz) with 32 GB of RAM.

\textbf{Architecture.} We train two small fully-connected neural network
(``dense-1'' and ``dense-3'') whose SDP relaxations can be quickly
solved using MOSEK~\citep{mosek2019mosek}, and a larger convolutional
network (``CNN'') whose SDP relaxation must be solved using a custom
algorithm described below. The ``dense-1'' and ``dense-3'' models
respectively have one and three fully-connected layer(s) of 50 neurons,
and are trained on a $4\times4$ maxpooled version of the training
set (each image is downsampled to $7\times7$). The ``CNN'' model
has two convolutional layers (stride 2) with 16 and 32 filters (size
$4\times4$) respectively, followed by a fully-connected layer with
100 neurons, and is trained on the original dataset of $28\times28$
images. All models are implemented in tensorflow and trained over
50 epochs using the SGD optimizer (learning rate 0.01, momentum 0.9,
``Nesterov'' true). 

\textbf{Rank-2 Burer-Monteiro algorithm (``BM2'').} We use a rank-2
Burer--Monteiro algorithm to solve instances of the SDP relaxation
on the ``CNN'' model, by applying a local optimization algorithm
to the following (see Appendix~\ref{sec:BM2} for a detailed derivation
and implementation details)
\begin{align}
\min_{\u_{k},\v_{k}\in\R^{n}}\quad & \|\u_{0}-\hat{\x}\|^{2}+\|\v_{0}\|^{2}\tag{BM2}\label{eq:bm2}\\
\text{s.t.}\quad & \diag(\u_{k+1}\u_{k+1}^{T}+\v_{k+1}\v_{k+1}^{T})\le\diag((\Wmat_{k}\u_{k}+\bvec_{k})\u_{k+1}^{T}+\Wmat_{k}\v_{k+1}\v_{k+1}^{T})\nonumber \\
 & \u_{k+1}\ge\max\left\{ 0,\Wmat_{k}\u_{k}+\bvec_{k}\right\} ,\qquad\|\u_{\ell}-\hat{\z}\|^{2}+\|\v_{\ell}\|^{2}\le\rho^{2}\qquad\forall k.\nonumber 
\end{align}
Let $\{\u_{k}^{\star},\v_{k}^{\star}\}$ be a locally optimal solution
satisfying the first- and second-order optimality conditions (see
e.g.~\citep[Chapter~12]{nocedal2006numerical}). If $\v_{0}^{\star}=0$,
then by induction $\u_{k+1}^{\star}=\relu(\Wmat_{k}\u_{k}^{\star}+\bvec_{k})$
and $\v_{k}^{\star}=0$ for all $k$. It then follows from a well-known
result of \citet{burer2005local} (see also~\citep{boumal2020deterministic,journee2010low}
and in particular \citep[Lemma~1]{cifuentes2019burer}) that $\{\u_{k}^{\star},\v_{k}^{\star}\}$
corresponds to a rank-1 solution of the SDP relaxation, and is therefore
\emph{globally optimal}. Of course, such a solution must not exist
if the relaxation is loose; even when it does exist, the algorithm
might still fail to find it if it gets stuck in a \emph{spurious}
local minimum with $\v_{0}^{\star}\ne0$. Our experience is that the
algorithm consistently succeeds whenever the relaxation is tight,
but admittedly this is not guaranteed.

\begin{figure}[b]
\centering
\mbox{
\input{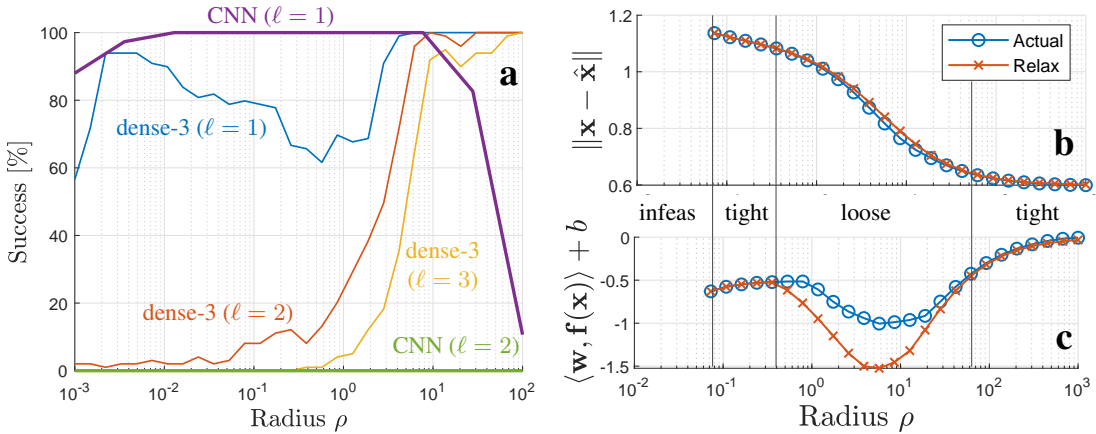}}

\caption{\textbf{\label{fig:experiments}a.} The SDP relaxation for problem
(\ref{eq:probB}) is generally tight over a single layer, and loose
over multiple layers. \textbf{b.} Viewing problem (\ref{eq:probA})
as (\ref{eq:probB}) taken at the limit $\rho\to\infty$, the resulting
SDP relaxation can be close to, but not exactly, tight, for finite
values of $\rho$. \textbf{c.} The SDP relaxation of (\ref{eq:probB})
can produce a near-optimal attack $\protect\x$ satisfying $\langle\protect\b w,\protect\b f(\protect\x)\rangle+b<0$
for problem (\ref{eq:probA}), even when relaxation itself is not
actually tight. }
\end{figure}

\textbf{Tightness for problem (\ref{eq:probB}). }Our theoretical
results suggest that the SDP relaxation for problem (\ref{eq:probB})
should be tight for one layer and loose for multiple layers. To verify,
we consider the first $k$ layers of the ``dense-3'' and ``CNN''
models over a range of radii $\rho$. In each case, we solve 1000
instances of the SDP relaxation, setting $\hat{\x}$ to be a new image
from the verification set, and $\hat{\z}=\b f(\u)$ where $\u$ is
the \emph{previous} image used as $\hat{\x}$. MOSEK solved each instance
of ``dense-3'' in 5-20 minutes and BM2 solved each instance of ``CNN''
in 15-60 minutes. We mark $\G^{\star}$ as numerically rank-1 if $\lambda_{1}(\G^{\star})/\lambda_{2}(\G^{\star})>10^{3}$,
and plot the success rates in Figure~\ref{fig:experiments}a. Consistent
with Theorem~\ref{thm:layer}, the relaxation over one layer is most
likely to be loose for intermediate values of $\rho$. Consistent
with Corollary~\ref{cor:multilayer}, the relaxation eventually becomes
tight once $\rho$ is large enough to yield a trivial solution. The
results for CNN are dramatic, with an 100\% success rate over a single
layer, and a 0\% success rate for two (and more) layers. BM2 is less
successful for very large and very small $\rho$ in part due to numerical
issues associated with the factor-of-two exponent in $\|\z-\hat{\z}\|^{2}\le\rho^{2}$.

\textbf{Application to problem (\ref{eq:probA}). }Viewing problem
(\ref{eq:probA}) as problem (\ref{eq:probB}) in the limit $\rho\to\infty$,
we consider a finite range of values for $\rho$, and solve the corresponding
SDP relaxation with $\hat{\z}=-\b w(b/\|\b w\|^{2}+\rho/\|\b w\|)$.
Here, the SDP relaxation is generally loose, so BM2 does not usually
succeed, and we must resort to using MOSEK to solve it on the small
``dense-1'' model. Figure~\ref{fig:experiments}b compares the
relaxation objective $\sqrt{\tr(\X)-2\langle\hat{\x},\x\rangle+\|\hat{\x}\|^{2}}$
with the actual distance $\|\x-\hat{\x}\|$, while Figure~\ref{fig:experiments}c
compares the feasibility predicted by the relaxation $\langle\b w,\z\rangle+b$
with the actual feasibility $\langle\b w,\b f(\x)\rangle+b$. The
relaxation is tight for $0.07\le\rho\le0.4$ and $\rho\ge60$ so the
plots coincide. The relaxation is loose for $0.4\le\rho\le60$, and
the relaxation objective is strictly greater than the actual distance
because $\X\succ\x\x^{T}$. The resulting attack $\x$ must fail to
satisfy $\|\b f(\x)-\hat{\z}\|\le\rho$, but in this case it is still
\emph{always} feasible for problem (\ref{eq:probA}). For $\rho<0.07$,
the SDP relaxation is infeasible, so we deduce that the output target
$\hat{\z}$ is not actually feasible. 

\section{Conclusions}

This paper presented a geometric study of the SDP relaxation of the
ReLU. We split the a modification of the robustness certification
problem into the composition of a convex projection onto a spherical
cap, and a nonconvex projection onto a hyperboloidal cap, so that
the relaxation is tight if and only if the latter projection lies
on the major axis of the hyperboloid. This insight allowed us to completely
characterize the tightness of the SDP relaxation over a single neuron,
and partially characterize the case for the single layer. The multilayer
case is usually loose due to the underlying hyperbolic geometry, and
this implies looseness in the SDP relaxation of the original certification
problem. Our rank-2 Burer-Monteiro algorithm was able to solve the
SDP relaxation on a convolution neural network, but better algorithms
are still needed before models of realistic scales can be certified.

\section*{Broader Impact}

This work contributes towards making neural networks more robust to
adversarial examples. This is a crucial roadblock before neural networks
can be widely adopted in safety-critical applications like self-driving
cars and smart grids. The ultimate, overarching goal is to take the
high performance of neural networks---already enjoyed by applications
in computer vision and natural language processing---and extend towards
applications in societal infrastructure. 

Towards this direction, SDP relaxations allow us to make mathematical
guarantees on the robustness of a given neural network model. However,
a blind reliance on mathematical guarantees leads to a false sense
of security. While this work contributes towards robustness of neural
networks, much more work is needed to understand the appropriateness
of neural networks for societal applications in the first place. 

\section*{Acknowledgments}

The author is grateful to Salar Fattahi, Cedric Josz, and Yi Ouyang
for early discussions and detailed feedback on several versions of
the draft. Partial financial support was provided by the National
Science Foundation under award ECCS-1808859.

\bibliographystyle{unsrtnat}
\bibliography{robustcnn}

\begin{thebibliography}{62}
\providecommand{\natexlab}[1]{#1}
\providecommand{\url}[1]{\texttt{#1}}
\expandafter\ifx\csname urlstyle\endcsname\relax
  \providecommand{\doi}[1]{doi: #1}\else
  \providecommand{\doi}{doi: \begingroup \urlstyle{rm}\Url}\fi

\bibitem[Biggio et~al.(2013)Biggio, Corona, Maiorca, Nelson, Srndic, Laskov,
  Giacinto, and Roli]{biggio2013evasion}
Battista Biggio, Igino Corona, Davide Maiorca, Blaine Nelson, Nedim Srndic,
  Pavel Laskov, Giorgio Giacinto, and Fabio Roli.
\newblock Evasion attacks against machine learning at test time.
\newblock In \emph{Joint European conference on machine learning and knowledge
  discovery in databases}, pages 387--402. Springer, 2013.

\bibitem[Szegedy et~al.(2014)Szegedy, Zaremba, Sutskever, Bruna, Erhan,
  Goodfellow, and Fergus]{Szegedy2014intriguing}
Christian Szegedy, Wojciech Zaremba, Ilya Sutskever, Joan Bruna, Dumitru Erhan,
  Ian Goodfellow, and Rob Fergus.
\newblock Intriguing properties of neural networks.
\newblock In \emph{International Conference on Learning Representations}, 2014.
\newblock URL \url{http://arxiv.org/abs/1312.6199}.

\bibitem[Carlini and Wagner(2017)]{carlini2017towards}
Nicholas Carlini and David Wagner.
\newblock Towards evaluating the robustness of neural networks.
\newblock In \emph{2017 ieee symposium on security and privacy (sp)}, pages
  39--57. IEEE, 2017.

\bibitem[Goodfellow et~al.(2015)Goodfellow, Shlens, and Szegedy]{43405}
Ian Goodfellow, Jonathon Shlens, and Christian Szegedy.
\newblock Explaining and harnessing adversarial examples.
\newblock In \emph{International Conference on Learning Representations}, 2015.
\newblock URL \url{http://arxiv.org/abs/1412.6572}.

\bibitem[Kurakin et~al.(2016)Kurakin, Goodfellow, and
  Bengio]{kurakin2016adversarial}
Alexey Kurakin, Ian Goodfellow, and Samy Bengio.
\newblock Adversarial machine learning at scale.
\newblock \emph{arXiv preprint arXiv:1611.01236}, 2016.

\bibitem[Madry et~al.(2018)Madry, Makelov, Schmidt, Tsipras, and
  Vladu]{madry2017towards}
Aleksander Madry, Aleksandar Makelov, Ludwig Schmidt, Dimitris Tsipras, and
  Adrian Vladu.
\newblock Towards deep learning models resistant to adversarial attacks.
\newblock In \emph{International Conference on Learning Representations}, 2018.

\bibitem[Katz et~al.(2017)Katz, Barrett, Dill, Julian, and
  Kochenderfer]{katz2017reluplex}
Guy Katz, Clark Barrett, David~L Dill, Kyle Julian, and Mykel~J Kochenderfer.
\newblock Reluplex: {A}n efficient {SMT} solver for verifying deep neural
  networks.
\newblock In \emph{International Conference on Computer Aided Verification},
  pages 97--117. Springer, 2017.

\bibitem[Ehlers(2017)]{ehlers2017formal}
Ruediger Ehlers.
\newblock Formal verification of piece-wise linear feed-forward neural
  networks.
\newblock In \emph{International Symposium on Automated Technology for
  Verification and Analysis}, pages 269--286. Springer, 2017.

\bibitem[Huang et~al.(2017)Huang, Kwiatkowska, Wang, and Wu]{huang2017safety}
Xiaowei Huang, Marta Kwiatkowska, Sen Wang, and Min Wu.
\newblock Safety verification of deep neural networks.
\newblock In \emph{International Conference on Computer Aided Verification},
  pages 3--29. Springer, 2017.

\bibitem[Wong and Kolter(2018)]{wong2018provable}
Eric Wong and Zico Kolter.
\newblock Provable defenses against adversarial examples via the convex outer
  adversarial polytope.
\newblock In \emph{International Conference on Machine Learning}, pages
  5286--5295, 2018.

\bibitem[Wong et~al.(2018)Wong, Schmidt, Metzen, and Kolter]{wong2018scaling}
Eric Wong, Frank Schmidt, Jan~Hendrik Metzen, and J~Zico Kolter.
\newblock Scaling provable adversarial defenses.
\newblock In \emph{Advances in Neural Information Processing Systems}, pages
  8400--8409, 2018.

\bibitem[Dvijotham et~al.(2018{\natexlab{a}})Dvijotham, Stanforth, Gowal, Mann,
  and Kohli]{dvijotham2018dual}
Krishnamurthy Dvijotham, Robert Stanforth, Sven Gowal, Timothy~A Mann, and
  Pushmeet Kohli.
\newblock A dual approach to scalable verification of deep networks.
\newblock In \emph{UAI}, volume~1, page~2, 2018{\natexlab{a}}.

\bibitem[Dvijotham et~al.(2018{\natexlab{b}})Dvijotham, Gowal, Stanforth,
  Arandjelovic, O'Donoghue, Uesato, and Kohli]{dvijotham2018training}
Krishnamurthy Dvijotham, Sven Gowal, Robert Stanforth, Relja Arandjelovic,
  Brendan O'Donoghue, Jonathan Uesato, and Pushmeet Kohli.
\newblock Training verified learners with learned verifiers.
\newblock \emph{arXiv preprint arXiv:1805.10265}, 2018{\natexlab{b}}.

\bibitem[Mirman et~al.(2018)Mirman, Gehr, and Vechev]{mirman2018differentiable}
Matthew Mirman, Timon Gehr, and Martin Vechev.
\newblock Differentiable abstract interpretation for provably robust neural
  networks.
\newblock In \emph{International Conference on Machine Learning}, pages
  3578--3586, 2018.

\bibitem[Singh et~al.(2018)Singh, Gehr, Mirman, P{\"u}schel, and
  Vechev]{singh2018fast}
Gagandeep Singh, Timon Gehr, Matthew Mirman, Markus P{\"u}schel, and Martin
  Vechev.
\newblock Fast and effective robustness certification.
\newblock In \emph{Advances in Neural Information Processing Systems}, pages
  10802--10813, 2018.

\bibitem[Gowal et~al.(2018)Gowal, Dvijotham, Stanforth, Bunel, Qin, Uesato,
  Arandjelovic, Mann, and Kohli]{gowal2018effectiveness}
Sven Gowal, Krishnamurthy Dvijotham, Robert Stanforth, Rudy Bunel, Chongli Qin,
  Jonathan Uesato, Relja Arandjelovic, Timothy Mann, and Pushmeet Kohli.
\newblock On the effectiveness of interval bound propagation for training
  verifiably robust models.
\newblock \emph{arXiv preprint arXiv:1810.12715}, 2018.

\bibitem[Zhang et~al.(2018)Zhang, Weng, Chen, Hsieh, and Daniel]{NIPS2018_7742}
Huan Zhang, Tsui-Wei Weng, Pin-Yu Chen, Cho-Jui Hsieh, and Luca Daniel.
\newblock Efficient neural network robustness certification with general
  activation functions.
\newblock In S.~Bengio, H.~Wallach, H.~Larochelle, K.~Grauman, N.~Cesa-Bianchi,
  and R.~Garnett, editors, \emph{Advances in Neural Information Processing
  Systems 31}, pages 4939--4948. Curran Associates, Inc., 2018.
\newblock URL
  \url{http://papers.nips.cc/paper/7742-efficient-neural-network-robustness-certification-with-general-activation-functions.pdf}.

\bibitem[Weng et~al.(2018)Weng, Zhang, Chen, Song, Hsieh, Daniel, Boning, and
  Dhillon]{weng2018towards}
Lily Weng, Huan Zhang, Hongge Chen, Zhao Song, Cho-Jui Hsieh, Luca Daniel,
  Duane Boning, and Inderjit Dhillon.
\newblock Towards fast computation of certified robustness for {ReLU} networks.
\newblock In \emph{International Conference on Machine Learning}, pages
  5276--5285, 2018.

\bibitem[Sinha et~al.(2018)Sinha, Namkoong, and Duchi]{sinha2017certifiable}
Aman Sinha, Hongseok Namkoong, and John Duchi.
\newblock Certifiable distributional robustness with principled adversarial
  training.
\newblock \emph{arXiv preprint arXiv:1710.10571}, 2, 2018.

\bibitem[Raghunathan et~al.(2018{\natexlab{a}})Raghunathan, Steinhardt, and
  Liang]{raghunathan2018certified}
Aditi Raghunathan, Jacob Steinhardt, and Percy Liang.
\newblock Certified defenses against adversarial examples.
\newblock In \emph{International Conference on Learning Representations},
  2018{\natexlab{a}}.

\bibitem[Raghunathan et~al.(2018{\natexlab{b}})Raghunathan, Steinhardt, and
  Liang]{NIPS2018_8285}
Aditi Raghunathan, Jacob Steinhardt, and Percy~S Liang.
\newblock Semidefinite relaxations for certifying robustness to adversarial
  examples.
\newblock In S.~Bengio, H.~Wallach, H.~Larochelle, K.~Grauman, N.~Cesa-Bianchi,
  and R.~Garnett, editors, \emph{Advances in Neural Information Processing
  Systems 31}, pages 10900--10910. Curran Associates, Inc., 2018{\natexlab{b}}.

\bibitem[Lov{\'a}sz and Schrijver(1991)]{lovasz1991cones}
L{\'a}szl{\'o} Lov{\'a}sz and Alexander Schrijver.
\newblock Cones of matrices and set-functions and 0--1 optimization.
\newblock \emph{SIAM journal on optimization}, 1\penalty0 (2):\penalty0
  166--190, 1991.

\bibitem[Lasserre(2001{\natexlab{a}})]{lasserre2001explicit}
Jean~B Lasserre.
\newblock An explicit exact sdp relaxation for nonlinear 0-1 programs.
\newblock In \emph{International Conference on Integer Programming and
  Combinatorial Optimization}, pages 293--303. Springer, 2001{\natexlab{a}}.

\bibitem[Lasserre(2001{\natexlab{b}})]{lasserre2001global}
Jean~B Lasserre.
\newblock Global optimization with polynomials and the problem of moments.
\newblock \emph{SIAM Journal on optimization}, 11\penalty0 (3):\penalty0
  796--817, 2001{\natexlab{b}}.

\bibitem[Parrilo(2003)]{parrilo2003semidefinite}
Pablo~A Parrilo.
\newblock Semidefinite programming relaxations for semialgebraic problems.
\newblock \emph{Mathematical programming}, 96\penalty0 (2):\penalty0 293--320,
  2003.

\bibitem[Cand{\`e}s and Recht(2009)]{candes2009exact}
Emmanuel~J Cand{\`e}s and Benjamin Recht.
\newblock Exact matrix completion via convex optimization.
\newblock \emph{Foundations of Computational mathematics}, 9\penalty0
  (6):\penalty0 717, 2009.

\bibitem[Cand{\`e}s and Tao(2010)]{candes2010power}
Emmanuel~J Cand{\`e}s and Terence Tao.
\newblock The power of convex relaxation: Near-optimal matrix completion.
\newblock \emph{IEEE Transactions on Information Theory}, 56\penalty0
  (5):\penalty0 2053--2080, 2010.

\bibitem[Recht et~al.(2010)Recht, Fazel, and Parrilo]{recht2010guaranteed}
Benjamin Recht, Maryam Fazel, and Pablo~A Parrilo.
\newblock Guaranteed minimum-rank solutions of linear matrix equations via
  nuclear norm minimization.
\newblock \emph{SIAM review}, 52\penalty0 (3):\penalty0 471--501, 2010.

\bibitem[Tjeng et~al.(2019)Tjeng, Xiao, and Tedrake]{tjeng2017evaluating}
Vincent Tjeng, Kai Xiao, and Russ Tedrake.
\newblock Evaluating robustness of neural networks with mixed integer
  programming.
\newblock In \emph{International Conference on Learning Representations}, 2019.

\bibitem[Lecuyer et~al.(2019)Lecuyer, Atlidakis, Geambasu, Hsu, and
  Jana]{lecuyer2019certified}
Mathias Lecuyer, Vaggelis Atlidakis, Roxana Geambasu, Daniel Hsu, and Suman
  Jana.
\newblock Certified robustness to adversarial examples with differential
  privacy.
\newblock In \emph{2019 IEEE Symposium on Security and Privacy (SP)}, pages
  656--672. IEEE, 2019.

\bibitem[Cohen et~al.(2019)Cohen, Rosenfeld, and Kolter]{cohen2019certified}
Jeremy Cohen, Elan Rosenfeld, and Zico Kolter.
\newblock Certified adversarial robustness via randomized smoothing.
\newblock In \emph{International Conference on Machine Learning}, pages
  1310--1320, 2019.

\bibitem[Salman et~al.(2019)Salman, Li, Razenshteyn, Zhang, Zhang, Bubeck, and
  Yang]{salman2019provably}
Hadi Salman, Jerry Li, Ilya Razenshteyn, Pengchuan Zhang, Huan Zhang, Sebastien
  Bubeck, and Greg Yang.
\newblock Provably robust deep learning via adversarially trained smoothed
  classifiers.
\newblock In \emph{Advances in Neural Information Processing Systems}, pages
  11289--11300, 2019.

\bibitem[Goemans and Williamson(1995)]{goemans1995improved}
Michel~X Goemans and David~P Williamson.
\newblock Improved approximation algorithms for maximum cut and satisfiability
  problems using semidefinite programming.
\newblock \emph{Journal of the ACM (JACM)}, 42\penalty0 (6):\penalty0
  1115--1145, 1995.

\bibitem[Frieze and Jerrum(1997)]{frieze1997improved}
Alan Frieze and Mark Jerrum.
\newblock Improved approximation algorithms for maxk-cut and max bisection.
\newblock \emph{Algorithmica}, 18\penalty0 (1):\penalty0 67--81, 1997.

\bibitem[Nesterov(1997)]{nesterov1997quality}
Yurii Nesterov.
\newblock Quality of semidefinite relaxation for nonconvex quadratic
  optimization.
\newblock Technical report, Universit{\'e} catholique de Louvain, Center for
  Operations Research and Econometrics, 1997.

\bibitem[So et~al.(2007)So, Zhang, and Ye]{so2007approximating}
Anthony Man-Cho So, Jiawei Zhang, and Yinyu Ye.
\newblock On approximating complex quadratic optimization problems via
  semidefinite programming relaxations.
\newblock \emph{Mathematical Programming}, 110\penalty0 (1):\penalty0 93--110,
  2007.

\bibitem[Candes and Plan(2010)]{candes2010matrix}
Emmanuel~J Candes and Yaniv Plan.
\newblock Matrix completion with noise.
\newblock \emph{Proceedings of the IEEE}, 98\penalty0 (6):\penalty0 925--936,
  2010.

\bibitem[Recht(2011)]{JMLR:v12:recht11a}
Benjamin Recht.
\newblock A simpler approach to matrix completion.
\newblock \emph{Journal of Machine Learning Research}, 12\penalty0
  (104):\penalty0 3413--3430, 2011.
\newblock URL \url{http://jmlr.org/papers/v12/recht11a.html}.

\bibitem[Fazlyab et~al.(2019)Fazlyab, Morari, and Pappas]{fazlyab2019safety}
Mahyar Fazlyab, Manfred Morari, and George~J Pappas.
\newblock Safety verification and robustness analysis of neural networks via
  quadratic constraints and semidefinite programming.
\newblock \emph{arXiv preprint arXiv:1903.01287}, 2019.

\bibitem[Dvijotham et~al.(2020)Dvijotham, Stanforth, Gowal, Qin, De, and
  Kohli]{dvijotham2020efficient}
Krishnamurthy~Dj Dvijotham, Robert Stanforth, Sven Gowal, Chongli Qin, Soham
  De, and Pushmeet Kohli.
\newblock Efficient neural network verification with exactness
  characterization.
\newblock In \emph{Uncertainty in Artificial Intelligence}, pages 497--507.
  PMLR, 2020.

\bibitem[Netrapalli et~al.(2013)Netrapalli, Jain, and
  Sanghavi]{netrapalli2013phase}
Praneeth Netrapalli, Prateek Jain, and Sujay Sanghavi.
\newblock Phase retrieval using alternating minimization.
\newblock In \emph{Advances in Neural Information Processing Systems}, pages
  2796--2804, 2013.

\bibitem[Chen and Candes(2015)]{chen2015solving}
Yuxin Chen and Emmanuel Candes.
\newblock Solving random quadratic systems of equations is nearly as easy as
  solving linear systems.
\newblock In \emph{Advances in Neural Information Processing Systems}, pages
  739--747, 2015.

\bibitem[Bhojanapalli et~al.(2016)Bhojanapalli, Neyshabur, and
  Srebro]{bhojanapalli2016global}
Srinadh Bhojanapalli, Behnam Neyshabur, and Nati Srebro.
\newblock Global optimality of local search for low rank matrix recovery.
\newblock In \emph{Advances in Neural Information Processing Systems}, pages
  3873--3881, 2016.

\bibitem[Goldstein and Studer(2018)]{goldstein2018phasemax}
Tom Goldstein and Christoph Studer.
\newblock Phasemax: Convex phase retrieval via basis pursuit.
\newblock \emph{IEEE Transactions on Information Theory}, 64\penalty0
  (4):\penalty0 2675--2689, 2018.

\bibitem[Cai et~al.(2010)Cai, Cand{\`e}s, and Shen]{cai2010singular}
Jian-Feng Cai, Emmanuel~J Cand{\`e}s, and Zuowei Shen.
\newblock A singular value thresholding algorithm for matrix completion.
\newblock \emph{SIAM Journal on optimization}, 20\penalty0 (4):\penalty0
  1956--1982, 2010.

\bibitem[Keshavan et~al.(2010{\natexlab{a}})Keshavan, Montanari, and
  Oh]{keshavan2010matrixa}
Raghunandan~H Keshavan, Andrea Montanari, and Sewoong Oh.
\newblock Matrix completion from noisy entries.
\newblock \emph{Journal of Machine Learning Research}, 11\penalty0
  (Jul):\penalty0 2057--2078, 2010{\natexlab{a}}.

\bibitem[Keshavan et~al.(2010{\natexlab{b}})Keshavan, Montanari, and
  Oh]{keshavan2010matrixb}
Raghunandan~H Keshavan, Andrea Montanari, and Sewoong Oh.
\newblock Matrix completion from a few entries.
\newblock \emph{IEEE transactions on information theory}, 56\penalty0
  (6):\penalty0 2980--2998, 2010{\natexlab{b}}.

\bibitem[Sun and Luo(2016)]{sun2016guaranteed}
Ruoyu Sun and Zhi-Quan Luo.
\newblock Guaranteed matrix completion via non-convex factorization.
\newblock \emph{IEEE Transactions on Information Theory}, 62\penalty0
  (11):\penalty0 6535--6579, 2016.

\bibitem[Ge et~al.(2016)Ge, Lee, and Ma]{ge2016matrix}
Rong Ge, Jason~D Lee, and Tengyu Ma.
\newblock Matrix completion has no spurious local minimum.
\newblock In \emph{Advances in Neural Information Processing Systems}, pages
  2973--2981, 2016.

\bibitem[Burer et~al.(2002)Burer, Monteiro, and Zhang]{burer2002rank}
Samuel Burer, Renato~DC Monteiro, and Yin Zhang.
\newblock Rank-two relaxation heuristics for max-cut and other binary quadratic
  programs.
\newblock \emph{SIAM Journal on Optimization}, 12\penalty0 (2):\penalty0
  503--521, 2002.

\bibitem[Shor(1987)]{Shor19871}
N.Z. Shor.
\newblock Quadratic optimization problems.
\newblock \emph{Soviet journal of computer and systems sciences}, 25\penalty0
  (6):\penalty0 1--11, 1987.
\newblock URL
  \url{https://www.scopus.com/inward/record.uri?eid=2-s2.0-0023454146&partnerID=40&md5=04f73c1bb9b80f64beacc99d77e7588a}.
\newblock cited By 120.

\bibitem[Boyd et~al.(2004)Boyd, Boyd, and Vandenberghe]{boyd2004convex}
Stephen Boyd, Stephen~P Boyd, and Lieven Vandenberghe.
\newblock \emph{Convex optimization}.
\newblock Cambridge university press, 2004.

\bibitem[P{\'o}lik and Terlaky(2007)]{polik2007survey}
Imre P{\'o}lik and Tam{\'a}s Terlaky.
\newblock A survey of the s-lemma.
\newblock \emph{SIAM review}, 49\penalty0 (3):\penalty0 371--418, 2007.

\bibitem[{MOSEK}(2019)]{mosek2019mosek}
ApS {MOSEK}.
\newblock The {MOSEK} optimization toolbox for {MATLAB} manual, 2019.
\newblock URL \url{https://docs.mosek.com/9.0/toolbox.pdf}.

\bibitem[Nocedal and Wright(2006)]{nocedal2006numerical}
Jorge Nocedal and Stephen Wright.
\newblock \emph{Numerical optimization}.
\newblock Springer Science \& Business Media, 2006.

\bibitem[Burer and Monteiro(2005)]{burer2005local}
Samuel Burer and Renato~DC Monteiro.
\newblock Local minima and convergence in low-rank semidefinite programming.
\newblock \emph{Mathematical Programming}, 103\penalty0 (3):\penalty0 427--444,
  2005.

\bibitem[Boumal et~al.(2020)Boumal, Voroninski, and
  Bandeira]{boumal2020deterministic}
Nicolas Boumal, Vladislav Voroninski, and Afonso~S Bandeira.
\newblock Deterministic guarantees for {Burer-Monteiro} factorizations of
  smooth semidefinite programs.
\newblock \emph{Communications on Pure and Applied Mathematics}, 73\penalty0
  (3):\penalty0 581--608, 2020.

\bibitem[Journ{\'e}e et~al.(2010)Journ{\'e}e, Bach, Absil, and
  Sepulchre]{journee2010low}
Michel Journ{\'e}e, Francis Bach, P-A Absil, and Rodolphe Sepulchre.
\newblock Low-rank optimization on the cone of positive semidefinite matrices.
\newblock \emph{SIAM Journal on Optimization}, 20\penalty0 (5):\penalty0
  2327--2351, 2010.

\bibitem[Cifuentes(2019)]{cifuentes2019burer}
Diego Cifuentes.
\newblock Burer-monteiro guarantees for general semidefinite programs.
\newblock \emph{arXiv preprint arXiv:1904.07147}, 2019.

\bibitem[Lemon et~al.(2016)Lemon, So, Ye, et~al.]{lemon2016low}
Alex Lemon, Anthony Man-Cho So, Yinyu Ye, et~al.
\newblock Low-rank semidefinite programming: Theory and applications.
\newblock \emph{Foundations and Trends{\textregistered} in Optimization},
  2\penalty0 (1-2):\penalty0 1--156, 2016.

\bibitem[Fukuda et~al.(2001)Fukuda, Kojima, Murota, and
  Nakata]{fukuda2001exploiting}
Mituhiro Fukuda, Masakazu Kojima, Kazuo Murota, and Kazuhide Nakata.
\newblock Exploiting sparsity in semidefinite programming via matrix completion
  {I}: {G}eneral framework.
\newblock \emph{SIAM Journal on Optimization}, 11\penalty0 (3):\penalty0
  647--674, 2001.

\bibitem[Vandenberghe et~al.(2015)Vandenberghe, Andersen,
  et~al.]{vandenberghe2015chordal}
Lieven Vandenberghe, Martin~S Andersen, et~al.
\newblock Chordal graphs and semidefinite optimization.
\newblock \emph{Foundations and Trends{\textregistered} in Optimization},
  1\penalty0 (4):\penalty0 241--433, 2015.

\end{thebibliography}

\newpage{}

\appendix

\section{\label{sec:collinear}Uniqueness of a Rank-1 Solution}

Consider the rank-constrained semidefinite program
\begin{alignat}{2}
 & \underset{\x\in\R^{n},\;\X\in\R^{n\times n}}{\text{minimize}}\quad & \langle\b D,\X\rangle+\langle\b f,\x\rangle\label{eq:relaxC}\\
 & \text{subject to}\quad & \langle\b A_{i},\X\rangle+\langle\b b_{i},\x\rangle & \le c_{i}\quad\text{for all }i\in\{1,2,\ldots,m\},\nonumber \\
 &  & \begin{bmatrix}1 & \x^{T}\\
\x & \X
\end{bmatrix}\succeq0,\quad\rank(\X) & \le p,\nonumber 
\end{alignat}
and its corresponding nonconvex optimization interpretation
\begin{alignat}{2}
 & \underset{\v_{1},\v_{2},\ldots,\v_{n}\in\R^{p}}{\text{minimize}}\quad & \sum_{k=1}^{n}\sum_{j=1}^{n}\langle\v_{j},\v_{k}\rangle\langle\b D,\e_{k}\e_{j}^{T}\rangle+\sum_{j=1}^{n}\langle\e,\v_{j}\rangle\langle\b f,\e_{j}\rangle\label{eq:ncvxC}\\
 & \text{subject to}\quad & \sum_{k=1}^{n}\sum_{j=1}^{n}\langle\v_{j},\v_{k}\rangle\langle\b A_{i},\e_{k}\e_{j}^{T}\rangle+\sum_{j=1}^{n}\langle\e,\v_{j}\rangle\langle\b b_{i},\e_{j}\rangle & \le c_{i}\quad\text{for all }i\in\{1,2,\ldots,m\},\nonumber 
\end{alignat}
where $\e$ is an arbitrary, fixed unit vector satisfying $\|\e\|=1$.
Our main result in this section is that we can guarantee a rank-1
solution to (\ref{eq:relaxC}) to be \emph{unique}, and hence computable
via an interior-point method, by verifying that every solution to
(\ref{eq:ncvxC}) is collinear with the unit vector $\e$. 
\begin{defn}
Fix $\e\in\R^{p}$ with $\|\e\|=1$. We say that $\v\in\R^{p}$ is
\emph{collinear} or that it satisfies \emph{collinearity} if $|\langle\e,\v\rangle|=\|\v\|$. 
\end{defn}

\begin{thm}[Unique rank-1]
\label{thm:collinear}Fix $\e\in\R^{p}$ with $\|\e\|=1$, and write
$\mathcal{V}^{\star}$ as the resulting solution set associated with
(\ref{eq:ncvxC}). Then, problem (\ref{eq:relaxC}) has a unique solution
satisfying $\X^{\star}=\x^{\star}(\x^{\star})^{T}$ if and only if
$\v_{1}^{\star},\v_{2}^{\star},\ldots,\v_{n}^{\star}$ are collinear
for all $(\v_{1}^{\star},\v_{2}^{\star},\ldots,\v_{n}^{\star})\in\mathcal{V}^{\star}$. 
\end{thm}

We briefly defer the proof of Theorem~\ref{thm:collinear} to discuss
its consequences. In the case of ReLU constraints, if the input vectors
$\x_{j}$ are collinear, then the output vectors $\z_{i}$ are also
collinear. 
\begin{lem}[Propagation of collinearity]
\label{lem:tight}Fix $\e\in\R^{p}$ with $\|\e\|=1$. Under the
ReLU constraints $\langle\e,\z\rangle\ge\max\{0,\langle\e,\sumsm_{j}w_{j}\x_{j}\rangle\rangle$
and $\left\langle \z,\z\right\rangle \le\left\langle \z,\sumsm_{j}w_{j}\x_{j}\right\rangle $,
if $\x_{j}$ is collinear for all $j$, then $\z$ is also collinear.
\end{lem}

\begin{proof}
Let $\x_{j}=x_{j}\e$ for all $j$, and write $\alpha=\sumsm_{j}w_{j}x_{j}$
for clarity. Observe that $\left\langle \z,\z\right\rangle \le\left\langle \z,\sumsm_{j}w_{j}\x_{j}\right\rangle =\sumsm_{j}w_{j}x_{j}\left\langle \z,\e\right\rangle =\alpha\left\langle \z,\e\right\rangle $.
If $\alpha<0$, then $\left\langle \z,\e\right\rangle =0$ and $\left\langle \z,\e\right\rangle =0$
as claimed. If $\alpha\ge0$, then $\left\langle \e,\z\right\rangle \ge\left\langle \e,\sumsm_{j}w_{j}\x_{j}\right\rangle =\sumsm_{j}w_{j}x_{j}=\alpha$.
Combined with the above, this yields $\left\langle \z,\z\right\rangle \le\alpha\left\langle \z,\e\right\rangle \le\left\langle \z,\e\right\rangle ^{2}$.
We actually have $\left\langle \z,\z\right\rangle =\left\langle \z,\e\right\rangle ^{2}$
as claimed, because $\left\langle \z,\z\right\rangle \ge\left\langle \z,\e\right\rangle ^{2}$
already holds by the Cauchy--Schwarz inequality. 
\end{proof}
Hence, the conditions for the uniquess of the rank-1 solution throughout
the main body of the paper are simply special cases of Theorem~\ref{thm:collinear}.

\begin{proof}[Proof of Lemma~\ref{lem:collinear1}]
Observe that the semidefinite program (\ref{eq:relaxB1}) is a special
instance of (\ref{eq:relaxC}), and that its nonconvex interpretation
(\ref{eq:ncvxB1}) is the corresponding instance of (\ref{eq:ncvxC}).
In the one-neuron case, Lemma~\ref{lem:tight} says that if $\x^{\star}$
satisfies collinearity, then $\z^{\star}$ also satisfies collinearity.
Or put in another way, $\x^{\star}$ and $\z^{\star}$ satisfy collinearity
if and only if $\x^{\star}$ satisfies collinearity. Using the latter
as an equivalent condition for the former and substituting into Theorem~\ref{thm:collinear}
yields Lemma~\ref{lem:collinear1} as desired.
\end{proof}
\begin{proof}[Proof of Lemma~\ref{lem:collinear2}]
We repeat the proof of Lemma~\ref{lem:collinear1}, but note that
$\x_{j}^{\star}$ and $\z_{i}^{\star}$ satisfy collinearity for all
$i$ and $j$ if and only if $\x_{j}^{\star}$ satisfies collinearity
for all $j$. Using the latter as an equivalent condition for the
former and substituting into Theorem~\ref{thm:collinear} yields
Lemma~\ref{lem:collinear2} as desired.
\end{proof}
The main intuition behind the proof of Theorem~\ref{thm:collinear}
is that a non-collinear solution to (\ref{eq:ncvxC}) corresponds
to a high rank solution to (\ref{eq:relaxC}) with $\rank(\X^{\star})>1$.
In turn, a rank-1 solution is unique if and only if there exists no
high-rank solutions; see~\citep[Theorem~2.4]{lemon2016low}. To make
these ideas rigorous, we begin by reviewing some preliminaries. First,
without loss of generality, we can fix $\e=\e_{1}$, that is, the
first canonical basis vector. If we wish to solve (\ref{eq:ncvxC})
with a different $\e=\e'$, then we simply need to find an orthonormal
matrix $\b U$ for which $\e'=\b U\e$, for example, using the Gram-Schmidt
process. Given a solution $\v_{1},\v_{2},\ldots,\v_{n}$ (\ref{eq:ncvxC})
with $\e=\e_{1}$, setting $\v_{j}'=\b U\v_{j}$ yields a solution
$\v_{1}',\v_{2}',\ldots,\v_{n}'$ to (\ref{eq:ncvxC}) with $\e=\e'$,
because $\langle\v_{k},\v_{j}\rangle=\langle\b U\v_{k},\b U\v_{j}\rangle=\langle\v_{k}',\v_{j}'\rangle$
and $\langle\e_{1},\v_{j}\rangle=\langle\b U\e_{1},\b U\v_{j}\rangle=\langle\e',\v_{j}'\rangle$.

The equivalence between (\ref{eq:relaxC}) and (\ref{eq:ncvxC}) is
established by using the solution to one problem to construct a \emph{corresponding
solution} satisfying the following relationship
\[
\langle\x,\e_{j}\rangle=\langle\e,\v_{j}\rangle,\qquad\langle\X,\e_{j}\e_{k}^{T}\rangle=\langle\v_{j},\v_{k}\rangle,
\]
for the other problem. In one direction, given a solution $\v_{1},\v_{2},\ldots,\v_{n}\in\R^{p}$
to (\ref{eq:ncvxC}), the corresponding solution to (\ref{eq:relaxC})
is simply 
\[
\x=[\langle\e,\v_{j}\rangle]_{j=1}^{n}=\b V^{T}\e,\qquad\X=[\langle\v_{j},\v_{k}\rangle]_{j,k=1}^{n}=\b V^{T}\b V,
\]
where $\b V=\begin{bmatrix}\v_{1} & \v_{2} & \cdots & \v_{n}\end{bmatrix}\in\R^{p\times n}$.
In the other direction, given a solution $\x$ and $\X$ to (\ref{eq:relaxC}),
we factorize $\X-\x\x^{T}=\tilde{\b V}^{T}\tilde{\b V}$ so that
\[
\begin{bmatrix}1 & \x^{T}\\
\x & \X
\end{bmatrix}=\begin{bmatrix}\e_{1} & \b V\end{bmatrix}^{T}\begin{bmatrix}\e_{1} & \b V\end{bmatrix},\quad\b V=\begin{bmatrix}\x^{T}\\
\tilde{\b V}
\end{bmatrix}=\begin{bmatrix}\v_{1} & \v_{2} & \cdots & \v_{n}\end{bmatrix}\in\R^{p\times n}.
\]
Then, $\v_{1},\v_{2},\ldots,\v_{n}$ is a corresponding solution to
(\ref{eq:ncvxC}) with $\e=\e_{1}$. 
\begin{proof}[Proof of Theorem~\ref{thm:collinear}]
($\Rightarrow$) Given a rank-1 solution $\X^{\star}=\x^{\star}(\x^{\star})^{T}$
of the relaxation (\ref{eq:relaxC}), we set $x_{j}^{\star}=\langle\e_{j},\x^{\star}\rangle$
and $\v_{j}^{\star}=\langle\e_{j},\x^{\star}\rangle\e$ to obtain
a corresponding solution $\v_{1}^{\star},\v_{2}^{\star},\ldots,\v_{n}^{\star}$
to (\ref{eq:ncvxC}) that satisfies collinearity. By contradiction,
suppose that there exists another solution $\v_{1}',\v_{2}',\ldots,\v_{n}'$
to (\ref{eq:ncvxC}) that does not satisfy collinearity, meaning that
there exists some $s$ such that $|\langle\e,\v_{s}'\rangle|\ne\|\v_{s}'\|$.
Then, its corresponding solution $\x',\X'$ is distinct from $\x^{\star},\X^{\star}$,
because $|\langle\e,\v_{s}'\rangle|\ne\|\v_{s}'\|$ but $|\langle\e,\v_{s}^{\star}\rangle|=\|\v_{s}^{\star}\|$,
so we can have either $\langle\X^{\star}-\X',\e_{s}\e_{s}^{T}\rangle=\|\v_{s}^{\star}\|^{2}-\|\v_{s}'\|^{2}=0$
or $\langle\x^{\star}-\x',\e_{s}\rangle=\langle\e,\v_{s}^{\star}-\v'_{s}\rangle=0$
but not both at the same time. This contradicts the hypothesis that
$\X^{\star}$ is a unique solution. 

($\Leftarrow$) Without loss of generality, we assume that $\e=\e_{1}$.
Given a solution $\v_{1}^{\star},\v_{2}^{\star},\ldots,\v_{n}^{\star}$
to (\ref{eq:ncvxC}) satisfying collinearity, we set $x_{j}^{\star}=\langle\e,\v_{j}^{\star}\rangle,$
$\x^{\star}=[x_{j}^{\star}]_{j=1}^{n},$ and $\X^{\star}=\x^{\star}(\x^{\star})^{T},$
in order to obtain a corresponding rank-1 solution to (\ref{eq:relaxC}).
By contradiction, suppose that there exists another solution $\x',\X'$
to (\ref{eq:relaxC}) that is distinct from $\x^{\star},\X^{\star}$,
with corresponding solution $\v_{1}',\v_{2}',\ldots,\v_{n}'$ to (\ref{eq:ncvxC}).
This solution $\v_{1}',\v_{2}',\ldots,\v_{n}'$ must satisfy collinearity,
or else our hypothesis is immediately violated. Under collinearity,
we again set $x_{j}'=\langle\e,\v_{j}'\rangle$ such that $\x'=[x_{j}']_{j=1}^{n}$
and $\X'=\x'(\x')^{T}$. Then, the following 
\[
\v_{j}=\frac{1}{2}\e_{1}(x_{j}^{\star}+x_{j}')+\frac{1}{2}\e_{2}(x_{j}^{\star}-x_{j}')
\]
yields another solution, since
\begin{align*}
\langle\e_{1},\v_{j}\rangle & =\frac{1}{2}(x_{j}^{\star}+x_{j}')=\frac{1}{2}(\langle\e_{1},\v_{j}^{\star}\rangle+\langle\e_{1},\v_{j}'\rangle)\\
\langle\v_{j},\v_{k}\rangle & =\frac{1}{4}(x_{j}^{\star}+x_{j}')(x_{k}^{\star}+x_{k}')+\frac{1}{4}(x_{j}^{\star}-x_{j}')(x_{k}^{\star}-x_{k}')=\frac{1}{2}(x_{j}^{\star}x_{k}^{\star}+x_{j}'x_{k}')\\
 & =\frac{1}{2}(\langle\v_{j}^{\star},\v_{k}^{\star}\rangle+\langle\v_{j}',\v_{k}'\rangle)
\end{align*}
In order for $\x',\X'$ is distinct from $\x^{\star},\X^{\star}$,
there must be some choice of $s$ such that $x_{s}^{\star}\ne x_{s}'$,
but this means that $\v_{s}$ does not satisfy collinearity, since
$\langle\e_{2},\v_{s}\rangle=\frac{1}{2}(x_{s}^{\star}-x_{s}')\ne0$.
This contradicts the hypothesis that all solutions $\v_{1},\v_{2},\ldots,\v_{n}$
to (\ref{eq:ncvxC}) satisfy collinearity.
\end{proof}

\section{\label{sec:reluproj}Projection onto ReLU Feasbility Set}

Fix $\e,\x\in\R^{p}$ and $\hat{z}\in\R$. Let $\alpha=\max\{\langle\e,\x\rangle,0\}$,
and define $\phi$ as the projection distance onto the spherical cap
defined by the ``ReLU feasible set'' (\ref{eq:proj_x}), restated
here as
\begin{equation}
\phi=\quad\min_{\z\in\R^{p}}\quad\|\z-\hat{z}\e\|\quad\text{s.t.}\quad\langle\e,\z\rangle\ge\alpha,\quad\|\z\|^{2}\le\langle\z,\x\rangle.\label{eq:minoverz-1}
\end{equation}
In the main text, we used intuitive, geometric arguments to prove
that
\begin{equation}
\phi=\begin{cases}
\alpha-\hat{z} & \hat{z}\le\alpha,\\
\|\hat{z}\e-\x/2\|-\|\x/2\| & \hat{z}>\alpha.
\end{cases}\label{eq:hypersec1}
\end{equation}
In this section, we will rigorously verify (\ref{eq:hypersec1}) and
then prove that the conditional statements are unnecessary, in that
$\phi$ simply takes on the larger of the two values, as in
\begin{equation}
\phi=\max\{\alpha-\hat{z},\quad\|\hat{z}\e-\x/2\|-\|\x/2\|\}.
\end{equation}
This was stated in the main text as Lemma~\ref{lem:reluproj}.

We first rigorously verify (\ref{eq:hypersec1}) by: 1) relaxing a
constraint for a specified case; 2) solving the relaxation in closed-form;
3) verifying that the closed-form solution satisfies the original
constraints, and must therefore be optimal for the original problem.
In the case of $\hat{z}\le\alpha$, the following relaxation
\[
\phi_{\lb1}=\min_{\z\in\R^{p}}\{\|\z-\hat{z}\e\|\quad:\quad\langle\e,\z\rangle\ge\alpha\}
\]
has solution $\z^{\star}=\alpha\e$ that is clearly feasible for (\ref{eq:minoverz-1})
since $\|\z^{\star}\|^{2}=\langle\z^{\star},\x\rangle=\alpha^{2}$.
Hence, this $\z^{\star}$ must be optimal; its objective $\|\z^{\star}-\hat{z}\e\|=\alpha-\hat{z}$
yields the desired value of $\phi$. 

In the case of $\hat{z}>\alpha$, the following relaxation
\[
\phi_{\lb2}=\min_{\z\in\R^{p}}\{\|\z-\hat{z}\e\|^{2}\quad:\quad\|\z\|^{2}\le\langle\z,\x\rangle\},
\]
must have an active constraint at optimality. Otherwise, the solution
would be $\z=\hat{z}\e$, but this cannot be feasible as $\hat{z}^{2}=\|\z\|\le\langle\z,\x\rangle=\hat{z}\langle\e,\x\rangle\le\hat{z}\alpha$
would contradict $\hat{z}>\alpha\ge0$. Applying Lagrange multipliers,
the solution reads $\z^{\star}=t\cdot\hat{z}\e+(1-t)\cdot\x/2$ where
$t=\|\x/2\|/\|\hat{z}\e-\x/2\|$ is chosen to make the constraint
active. We will need the following lemma to verify that $\langle\e,\z^{\star}\rangle\ge\alpha$.
\begin{lem}
\label{lem:feas_lem}Let $|v|\le R$. If $u>\sqrt{R^{2}-v^{2}}$,
then $Ru/\sqrt{u^{2}+v^{2}}\ge\sqrt{R^{2}-v^{2}}$.
\end{lem}

\begin{proof}
We will prove that if $u^{2}+v^{2}>R^{2}$ then $R^{2}u^{2}/(u^{2}+v^{2})+v^{2}\ge R^{2}$.
By contradiction, suppose that $R^{2}u^{2}/(u^{2}+v^{2})+v^{2}<R$.
If $u^{2}+v^{2}=0$, then the premise is already false. Otherwise,
we multiply by $u^{2}+v^{2}>0$ to yield $R^{2}u^{2}+v^{2}(u^{2}+v^{2})<R^{2}(u^{2}+v^{2}),$
or equivalently $v^{2}(u^{2}+v^{2}-R^{2})<0$. This last condition
is only possible if $v\ne0$ and $u^{2}+v^{2}<R^{2}$, but this again
contradicts the premise.
\end{proof}
For $u=2\hat{z}-\langle\e,\x\rangle,$ $v=\sqrt{\|\x\|^{2}-\langle\e,\x\rangle^{2}},$
and $R=\|\x\|$, observe that 
\begin{align*}
t & =\frac{\|\x/2\|}{\|\hat{z}\e-\x/2\|}=\frac{R}{\sqrt{u^{2}+v^{2}}}, & \alpha & =\max\{\langle\e,\x\rangle,0\}=\frac{\langle\e,\x\rangle}{2}+\frac{|\langle\e,\x\rangle|}{2}.
\end{align*}
Then, $\z^{\star}=t\cdot\hat{z}\e+(1-t)\cdot\x/2$ is feasible for
(\ref{eq:minoverz-1}), because substituting $u,v,R$ into Lemma~\ref{lem:feas_lem}
yields 
\[
\hat{z}>\alpha\quad\iff\quad\hat{z}-\frac{\langle\e,\x\rangle}{2}>\frac{|\langle\e,\x\rangle|}{2}\quad\implies\quad t\cdot\left(\hat{z}-\frac{\langle\e,\x\rangle}{2}\right)\ge\frac{|\langle\e,\x\rangle|}{2},
\]
and this in turn implies that 
\begin{align*}
\langle\e,\z^{\star}\rangle & =\frac{\langle\e,\x\rangle}{2}+t\cdot\left(\hat{z}-\frac{\langle\e,\x\rangle}{2}\right)\ge\frac{\langle\e,\x\rangle}{2}+\frac{|\langle\e,\x\rangle|}{2}=\alpha.
\end{align*}
Hence, this $\z^{\star}$ must be optimal; its objective $\|\z^{\star}-\hat{z}\e\|=(1-t)\|\hat{z}\e-\x/2\|$
yields the desired value of $\phi$. 

Finally, we prove (\ref{eq:hypersec}) by showing that the conditional
statements in (\ref{eq:hypersec1}) are unnecessary. 
\begin{proof}[Proof of Lemma~\ref{lem:reluproj}]
If $\hat{z}>\alpha$, then clearly $\phi=\phi_{\lb2}\ge0$ by construction,
but $\alpha-\hat{z}<0$, so $\phi=\max\{\alpha-\hat{z},\phi_{\lb2}\}$
as desired. For $\hat{z}\le\alpha$, we will proceed by examining
two cases. First, suppose that $\hat{z}\ge0$ and hence $\alpha=\langle\e,\x\rangle$
and $\hat{z}\le\langle\e,\x\rangle$. Then, $\|\hat{z}\e-\x/2\|^{2}-\|\x/2\|^{2}=\hat{z}(\hat{z}-\langle\e,\x\rangle)\le0,$
and $\|\hat{z}\e-\x/2\|-\|\x/2\|\le0,$ so $\phi=\max\{\phi_{\lb1},\|\hat{z}\e-\x/2\|-\|\x/2\|\}$
as desired. In the case of $\hat{z}\le0$, Lemma~\ref{lem:quadratics}
shows that $\langle\e,\x\rangle-\hat{z}\le\rho$ implies $\|\hat{z}\e-\x/2\|-\|\x/2\|\le\rho$,
since with $u=\langle\e,\x\rangle$, $v=\sqrt{\|\x\|^{2}-\langle\e,\x\rangle^{2}},$
$c=|\hat{z}|$, and $a=\rho$, we have
\[
\|\x/2+|\hat{z}|\e\|-\|\x/2\|\le\rho\quad\iff\quad\frac{\langle\e,\x\rangle+|\hat{z}|}{\rho}\le\sqrt{1+\frac{\|\x\|^{2}-\langle\e,\x\rangle^{2}}{\hat{z}^{2}-\rho^{2}}}
\]
but $\langle\e,\x\rangle-\hat{z}\le\rho$ already implies $\frac{1}{\rho}[\langle\e,\x\rangle+|\hat{z}|]\le1$.
In particular, the fact that $\langle\e,\x\rangle-\hat{z}\le\phi_{\lb1}$
implies $\|\hat{z}\e-\x/2\|-\|\x/2\|\le\phi_{\lb1}$ shows that we
have $\phi=\max\{\phi_{\lb1},\|\hat{z}\e-\x/2\|-\|\x/2\|\}$.
\end{proof}

\section{\label{sec:hypproj1}Projection onto a hyperbola}

Fix $\e,\x\in\R^{p}$ and $\hat{x},\hat{z},\rho\in\R$ such that $\hat{z}>\rho>0$.
Define $\psi$ as the projection distance onto the hyperboloidal cap
(\ref{eq:proj_x}), restated here
\begin{equation}
\psi=\quad\min_{\x\in\R^{p}}\quad\|\x-\hat{x}\e\|\quad\text{s.t.}\quad\langle\e,\x\rangle-\hat{z}\le\rho,\quad\|2\hat{z}\e-\x\|-\|\x\|\le2\rho.\label{eq:HP-1}
\end{equation}
Without loss of generality, we can fix $\e=\e_{1}$ (see Appendix~\ref{sec:collinear}),
and split the coordinates of $\x$ as in $u=\x[1]$ and $\v[j]=\x[1+j]$
for $j\in\{1,2,\ldots,p-1\}$ to rewrite (\ref{eq:HP-1}) as the following
\begin{align}
\psi^{2}=\quad\min_{(u,\v)\in\R^{p}} & \quad(u-\hat{x})^{2}+\|\v\|^{2}\label{eq:HP-2}\\
\text{s.t.} & \quad u-\hat{z}\le\rho,\quad\sqrt{(u-2\hat{z})^{2}+\|\v\|^{2}}-\sqrt{u^{2}+\|\v\|^{2}}\le2\rho.\nonumber 
\end{align}
Observe that the variable $\v\in\R^{p-1}$ only appears in (\ref{eq:HP-2})
via its norm $\|\v\|$. Hence, (\ref{eq:HP-2}) is equivalent to the
following problem
\begin{align}
\psi^{2}=\quad\min_{u,v\in\R} & \quad(u-\hat{x})^{2}+v^{2}\label{eq:HP-3}\\
\text{s.t.} & \quad u-\hat{z}\le\rho,\quad\sqrt{(u-2\hat{z})^{2}+v^{2}}-\sqrt{u^{2}+v^{2}}\le2\rho,\nonumber 
\end{align}
and a solution $\v^{\star}$ to (\ref{eq:HP-2}) can be recovered
from a solution $v^{\star}$ to (\ref{eq:HP-3}) by picking any unit
vector $\b s\in\R^{p-1}$ with $\|\b s\|=1$ and setting $\v^{\star}=v^{\star}\b s$.
We have reduced the projection over a hyperboloid (\ref{eq:HP-1})
into a projection onto a hyperbola (\ref{eq:HP-3}) by taking a quotient
over the minor-axis directions. To proceed, we will need the following
technical lemma, which is mechancially derived by completing the square
and collecting terms.
\begin{lem}
\label{lem:quadratics}Given semi-major axis $a>0$, semi-minor axis
$b>0$, and focus $c=\sqrt{a^{2}+b^{2}}$, the following hold\begin{subequations}
\begin{gather}
\sqrt{(u-2c)^{2}+v^{2}}-\sqrt{u^{2}+v^{2}}\le2a\quad\iff\quad\frac{u-c}{a}\ge\sqrt{1+\frac{v^{2}}{b^{2}}},\label{eq:hyperbpos}\\
\sqrt{(u+2c)^{2}+v^{2}}-\sqrt{u^{2}+v^{2}}\le2a\quad\iff\quad\frac{u+c}{a}\le\sqrt{1+\frac{v^{2}}{b^{2}}}.\label{eq:hyperbneg}
\end{gather}
\end{subequations}
\end{lem}

We use Lemma~\ref{lem:quadratics} to rewrite the hyperbolic constraint
in (\ref{eq:HP-3}) in quadratic form, as in
\begin{equation}
\psi^{2}=\min_{u,v\in\R}\quad(u-\hat{x})^{2}+v^{2}\quad\text{s.t.}\quad\frac{u-\hat{z}}{\rho}\le1,\quad\frac{(u-\hat{z})^{2}}{\rho^{2}}-\frac{v^{2}}{\hat{z}^{2}-\rho^{2}}\le1.\label{eq:HP-4}
\end{equation}
We will need the following to solve (\ref{eq:HP-4}). This is the
main result of this section.
\begin{thm}[Axial projection onto a hyperbola]
\label{thm:hyper1}The problem data $\b a,\x\in\R^{m},$ $\b c\in\R^{m}$
and $b\in\R$ satisfy
\[
\b a,\b c\ne0,\quad|\langle\b a,\x\rangle-b|-1<\|\b a\|^{2}/\|\b c\|^{2}
\]
if and only if the following projection
\begin{align*}
(\u^{\star},\v^{\star})= & \arg\min_{\u,\v}\left\{ \|\u-\x\|^{2}+\|\v\|^{2}:(\langle\b a,\u\rangle-b)^{2}-\langle\b c,\v\rangle^{2}\le1\right\} 
\end{align*}
has a unique solution
\begin{align*}
\b u^{\star} & =\x-\b a\frac{(\langle\b a,\x\rangle-b)}{\|\b a\|^{2}}\left(1-\frac{1}{|\langle\b a,\x\rangle-b|}\right),\\
\v^{\star} & =0.
\end{align*}
\end{thm}

The proof of Theorem~\ref{thm:hyper1} will span the remainder of
this section. Lemma~\ref{lem:hypproj1} is clearly a special instance
as applied to (\ref{eq:HP-4}).
\begin{proof}[Proof of Lemma~\ref{lem:hypproj1}]
If $\hat{x}\ge\hat{z}-\rho$, then relaxing the hyperbolic constraint
in (\ref{eq:HP-4}) yields a unique solution of $u^{\star}=\min\{\hat{x},\hat{z}+\rho\}$
and $v^{\star}=0$. Indeed, this solution also satisfies the hyperbolic
constraint, and is therefore optimal for (\ref{eq:HP-4}). Otherwise,
if $\hat{x}<\hat{z}-\rho$, then we will use relax the linear constraint
in (\ref{eq:HP-4}) and apply Theorem~\ref{thm:hyper1}. Here, $\b a=1/\rho$,
$b=\hat{z}/\rho$, $\b c=1/\sqrt{\hat{z}^{2}-\rho^{2}}$, and $\x=\hat{x}$,
and the condition for (\ref{eq:HP-4}) to have a unique condition
$u^{\star}$ and $v^{\star}$ with $v^{\star}=0$ is
\begin{equation}
|\hat{x}-\hat{z}|/\rho-1<(\hat{z}^{2}-\rho^{2})/\rho^{2}\qquad\iff\qquad|\hat{x}-\hat{z}|<\hat{z}^{2}/\rho.\label{eq:tmp}
\end{equation}
It is easy to verify that the resulting solution is feasible for (\ref{eq:HP-4}),
and hence optimal. Under the premise $\hat{x}-\hat{z}<-\rho<0$, the
condition (\ref{eq:tmp}) is just $\hat{x}>\hat{z}-\hat{z}^{2}/\rho$,
which also implies $\hat{x}\ge\hat{z}-\rho$ because $\hat{z}>\rho$.
Hence, we have covered both cases; the condition $(\hat{z}-\hat{x})<\hat{z}^{2}/\rho$
guarantees a unique $u^{\star}$ and $v^{\star}=0$ as claimed.
\end{proof}
We will now prove Theorem~\ref{thm:hyper1}. The Euclidean projection
onto a hyperbola is the minimization of one quadratic function subject
to another quadratic function. This is well-known to be a tractable
problem via the S-procedure (see e.g.~\citep[p.~655]{boyd2004convex}
or~\citep{polik2007survey}). In its original form, it states that
for two quadratics $f(\x)$ and $g(\x)$ for which there exists $\x_{0}$
satisfying $g(\x_{0})<0$, that
\[
f(\x)\ge0\quad\text{holds for all }\x\text{ satisfying }g(\x)\le0
\]
if and only if there exists $\lambda\ge0$ such that
\[
f(\x)+\lambda g(\x)\ge0\quad\text{holds for all }\x.
\]
Clearly, a corollary of the S-procedure is strong duality, as in
\[
\min_{\x}\{f(\x):g(\x)\le0)\quad=\quad\max_{\lambda\ge0}\min_{\x}\{f(\x)+\lambda g(\x)\},
\]
and so the Karush--Kuhn--Tucker conditions allow us to solve the
primal by solving the dual, assuming the existence of a strictly feasible
point $\x_{0}$ with $g(\x_{0})<0$. To proceed, we will need the
following technical lemma, which is mechancially derived by applying
the Sherman-Morrison identity.
\begin{lem}[Rank-1 update]
\label{lem:rank1up}Given $\b a,\x\in\R^{m},$ $b\in\R,$ and $\lambda>-1/\|\b a\|^{2}$,
the following projection 
\[
\u^{\star}=\arg\min_{u\in\R^{n}}\{\|\u-\x\|^{2}+\lambda(\langle\b a,\u\rangle-b)^{2}\}
\]
has a unique solution $\u^{\star}$ satisfying 
\begin{gather*}
\u^{\star}=\x-\lambda\b a\left(\frac{\langle\b a,\x\rangle-b}{1+\lambda\|\b a\|^{2}}\right),\qquad\langle\b a,\u^{\star}\rangle-b=\frac{\langle\b a,\x\rangle-b}{1+\lambda\|\b a\|^{2}}.\\
\|\u^{\star}-\x\|^{2}+\lambda(\langle\b a,\u^{\star}\rangle-b)^{2}=\frac{\lambda(\langle\b a,\x\rangle-b)^{2}}{1+\lambda\|\b a\|^{2}}
\end{gather*}
\end{lem}

We will actually solve the most general form of the projection problem. 
\begin{lem}[General projection onto a single hyperbola]
\label{lem:hyper1}Let $\b a,\x\in\R^{m},$ $\b c,\b y\in\R^{m}$
and $b,d\in\R$ satisfy $\b a,\b c\ne0$. Let $\u^{\star}\in\R^{m},$
$\v^{\star}\in\R^{m}$ be solutions to the projection
\begin{align*}
\phi= & \min_{\u,\v}\left\{ \|\u-\x\|^{2}+\|\v-\b y\|^{2}:(\langle\b a,\u\rangle-b)^{2}-(\langle\b c,\v\rangle-d)^{2}\le1\right\} ,
\end{align*}
and let $\lambda^{\star}$ be the unique solution to the Lagrangian
dual
\[
\phi_{\lb}=\max_{0\le\lambda\le1/\|\b c\|^{2}}\left\{ \lambda\left[\frac{(\langle\b a,\x\rangle-b)^{2}}{1+\lambda\|\b a\|^{2}}-\frac{(\langle\b c,\b y\rangle-d)^{2}}{1-\lambda\|\b c\|^{2}}-1\right]\right\} .
\]
Then, $\phi=\phi_{\lb}$. Moreover the primal solutions are unique
if and only if $\lambda^{\star}<1/\|\b c\|^{2}$, with values
\[
\u^{\star}=\x-\lambda^{\star}\b a\left(\frac{\langle\b a,\x\rangle-b}{1+\lambda^{\star}\|\b a\|^{2}}\right),\qquad v^{\star}=y+\lambda^{\star}\b c\left(\frac{\langle\b c,\b y\rangle-d}{1-\lambda^{\star}\|\b c\|^{2}}\right).
\]
\end{lem}

\begin{proof}
We define the following two quadratics and corresponding Lagrangian
\begin{gather*}
f(\u,\v)=\|\u-\x\|^{2}+\|\v-\b y\|^{2},\\
g(\u,\v)=(\langle\b a,\u\rangle-b)^{2}-(\langle\b c,\v\rangle-d)^{2}-1,\\
L(\u,\v,\lambda)=f(\u,\v)+\lambda g(\u,\v).
\end{gather*}
Note that $\u_{0}=b\b a/\|\b a\|$ and $\v_{0}=d\b c/\|\b c\|$ satisfies
$g(\u_{0},\v_{0})<0$, so strong duality holds via the S-procedure.
Next, we apply Lemma~\ref{lem:rank1up} to yield the Lagrangian dual
$\phi_{\lb}$ via
\[
\min_{\u,\v}L(\u,\v,\lambda)=\begin{cases}
\lambda\left[\frac{(\langle\b a,\x\rangle-b)^{2}}{1+\lambda\|\b a\|^{2}}-\frac{(\langle\b c,\b y\rangle-d)^{2}}{1-\lambda\|\b c\|^{2}}-1\right] & \lambda\le1/\|\b c\|^{2},\\
-\infty & \lambda>1/\|\b c\|^{2}.
\end{cases}
\]
It is easy to verify that the dual function above is strongly concave
over $\lambda$, so the solution $\lambda^{\star}$ is unique. Finally,
if $\lambda^{\star}<1/\|\b c\|^{2}$, then the Lagrangian $L(\u,\v,\lambda^{\star})$
is strongly convex, and the primal solutions $\u^{\star}$ and $\v^{\star}$
are both uniquely determined by minimizing $L(\u,\v,\lambda^{\star})$.
Otherwise, if $\lambda^{\star}=1/\|\b c\|^{2}$, then $L(\u,\v,\lambda^{\star})$
is weakly convex over $\v$. Here, $\u^{\star}$ is uniquely determined
by minimizing $L(\u,\v,\lambda^{\star})$, but $\v^{\star}$ can be
any choice that satisfies primal feasibility $(\langle\b a,\u^{\star}\rangle-b)^{2}-(\langle\b c,\v^{\star}\rangle-d)^{2}=1$,
and is therefore nonunique.
\end{proof}
Finally, we prove Theorem~\ref{thm:hyper1} using Lemma~\ref{lem:hyper1}. 
\begin{proof}[Proof of Theorem~\ref{thm:hyper1}]
The axial projection problem of Theorem~\ref{thm:hyper1} is an
instance of the more general projection problem in Lemma~\ref{lem:hyper1}
with $\b y=0$ and $d=0$. The intended claim holds so long as $\lambda^{\star}<1/\|\b c\|^{2}$.
Now, first order optimality in the Lagrangian dual reads
\[
\frac{(\langle\b a,\x\rangle-b)^{2}}{(1+\lambda^{\star}\|\b a\|^{2})^{2}}-\frac{(\langle\b c,\b y\rangle-d)^{2}}{(1-\lambda^{\star}\|\b c\|^{2})^{2}}-1=0,
\]
and this implies $1+\lambda^{\star}\|\b a\|^{2}=|\langle\b a,\x\rangle-b|$
and hence $\lambda^{\star}=(|\langle\b a,\x\rangle-b|-1)/\|\b a\|^{2}$.
Finally, imposing the bound $\lambda^{\star}<1/\|\b c\|^{2}$ on this
value yields our desired claim.
\end{proof}

\section{\label{sec:hypproj2}Projection onto several hyperbolas}

Given $\Wmat=[W_{i,j}]\in\R^{m\times n},$ $\hat{\x}=[\hat{x}_{j}]\in\R^{n},$
$\hat{\z}=[\hat{z}_{i}]\in\R^{m},$ $\e\in\R^{p},$ and $\rho_{i}$
satisfying $\hat{z}_{i}>\rho>0$, we will partially solve
\begin{equation}
\min_{\x_{j}\in\R^{p}}\quad\sumsm_{j}\|\x_{j}-\hat{x}_{j}\e\|^{2}\quad\text{s.t.}\quad\begin{array}{r}
\langle\e,\sumsm_{j}W_{i,j}\x_{j}\rangle-\hat{z}_{i}\le\rho_{i}\;\forall i,\\
\|\hat{z}_{i}\e-\sumsm_{j}W_{i,j}\x_{j}/2\|-\|\sumsm_{j}W_{i,j}\x_{j}/2\|\le\rho_{i}\;\forall i.
\end{array}\label{eq:HPP-1}
\end{equation}
Without loss of generality, we can fix $\e=\e_{1}$ and split the
coordinates of $\x_{j}$ as in $\u[j]=\x_{j}[1]$ for all $j$ and
$\v_{k}[j]=\x_{j}[1+k]$ for all $j,k$ to rewrite (\ref{eq:HPP-1})
as the following
\begin{align}
\min_{\u,\v_{j}\in\R^{n}}\quad & \|\u-\hat{\x}\|^{2}+\sumsm_{k}\|\v_{k}\|^{2}\label{eq:HPP-2}\\
\text{s.t.}\quad & \langle\b w_{i},\u\rangle-\hat{z}_{i}\le\rho_{i},\nonumber \\
 & \sqrt{(\langle\b w_{i},\u\rangle-2\hat{z}_{i})^{2}+\sumsm_{k}\langle\b w_{i},\v_{k}\rangle^{2}}-\sqrt{\langle\b w_{i},\u\rangle^{2}+\sumsm_{k}\langle\b w_{i},\v_{k}\rangle^{2}}\le2\rho_{i},\nonumber 
\end{align}
for all $i$, where $\b w_{i}[j]=\Wmat[i,j]$ is the $i$-th row of
$\Wmat$. Applying Lemma~\ref{lem:quadratics} then rewrites (\ref{eq:HPP-2})
as the following.
\begin{equation}
\min_{\u,\v_{j}\in\R^{n}}\quad\|\u-\hat{\x}\|^{2}+\sumsm_{k}\|\v_{k}\|^{2}\quad\text{s.t.}\quad\sqrt{1+\frac{\sumsm_{k}\langle\b w_{i},\v_{k}\rangle^{2}}{\hat{z}_{i}^{2}-\rho_{i}^{2}}}\le\frac{\langle\b w_{i},\u\rangle-\hat{z}}{\rho_{i}}\le1,\label{eq:HPP-3}
\end{equation}
We will need the following to solve (\ref{eq:HPP-3}). This is the
main result of this section.
\begin{thm}[Axial projection onto several hyperbolas]
\label{thm:hyper2}If the problem data $\x\in\R^{m},$ $\b a_{i}\in\R^{m}$,
$b_{i}\in\R,$ $\b c_{i}\in\R^{n}$ for $i\in\{1,2,\ldots,\ell\}$
satisfy
\[
\|\b C\|^{2}\cdot(\|(\b A\b A^{T})^{-1}(\b A\x-\b b)\|_{\infty}+\|(\b A\b A^{T})^{-1}\|_{\infty})<1
\]
where $\b A[i,j]=\b a_{i}[j]$, $\b b[i]=b_{i}$, and $\b C[i,j]=\b c_{i}[j]$
for all $i,j$, then the following projection
\begin{align*}
(\u^{\star},\v^{\star})= & \arg\min_{\u,\v}\left\{ \|\u-\x\|^{2}+\sumsm_{j}\|\v_{j}\|^{2}:(\langle\b a_{i},\u\rangle-b_{i})^{2}-\sumsm_{j}\langle\b c_{i},\v_{j}\rangle^{2}\le1\quad\forall i\right\} 
\end{align*}
has a unique solution $(\u^{\star},\v^{\star})$ with $\v_{j}^{\star}=0$.
\end{thm}

The proof of Theorem~\ref{thm:hyper2} will span the remainder of
this section. Lemma~\ref{lem:hypproj2} is clearly a special instance
as applied to (\ref{eq:HPP-3}).
\begin{proof}[Proof of Lemma~\ref{lem:hypproj2}]
Write $\b D_{1}=\diag(\rho_{i})$ and $\b D_{2}=\diag(\sqrt{\hat{z}_{i}^{2}-\rho_{i}^{2}})$.
Then, we apply Theorem~\ref{thm:hyper2} with $\x=\hat{\x}$, $\b A=\b D_{1}^{-1}\Wmat$,
$\b b=\b D_{1}^{-1}\hat{\z}$, and $\b C=\b D_{2}^{-1}\b W$. Clearly
\begin{gather*}
\|\b C\|^{2}=\|\b D_{2}^{-1}\b W\|^{2}\le\|\b W\|^{2}/(\hat{z}_{\min}^{2}-\rho_{\max}^{2})\\
\|(\b A\b A^{T})^{-1}(\b A\x-\b b)\|_{\infty}=\|\b D_{1}(\Wmat\Wmat^{T})^{-1}(\Wmat\hat{\x}-\hat{\z})\|_{\infty}\le\rho_{\max}\|(\Wmat\Wmat^{T})^{-1}(\Wmat\hat{\x}-\hat{\z})\|_{\infty}\\
\|\b D_{1}(\Wmat\Wmat^{T})^{-1}\b D_{1}\|_{\infty}\le\rho_{\max}^{2}\|(\Wmat\Wmat^{T})^{-1}\|_{\infty}
\end{gather*}
and hence the condition in Theorem~\ref{thm:hyper2} is the following
\[
\rho_{\max}\|\b W\|^{2}\|(\Wmat\Wmat^{T})^{-1}(\Wmat\hat{\x}-\hat{\z})\|_{\infty}+\rho_{\max}^{2}\|\b W\|^{2}\|(\Wmat\Wmat^{T})^{-1}\|_{\infty}<\hat{z}_{\min}^{2}-\rho_{\max}^{2}
\]
which is the same condition stated in Lemma~\ref{lem:hypproj2}.
\end{proof}
Our proof of Theorem~\ref{thm:hyper2} is based on a SDP relaxation.
\begin{proof}[Proof of Theorem~\ref{thm:hyper2}]
The problem is nonconvex over $\v$, but a convex relaxation is easily
constructed by representing the quadratic outer product $\sumsm_{k}\v_{k}\v_{k}^{T}$
by $\b V\succeq0$, as in 
\begin{alignat*}{2}
\underset{\u\in\R^{m},\v\in\R^{n}}{\text{minimize }}\quad & \frac{1}{2}\|\u-\x\|^{2}+\frac{1}{2}\tr(\b V)\\
\text{subject to }\quad & -1\le\langle\b a_{i},\u\rangle-b_{i}\le\sqrt{1+\langle\b c_{i}\b c_{i}^{T},\b V\rangle}\qquad & \text{for all }i\in\{1,2,\ldots,\ell\}
\end{alignat*}
with the relaxation being exact whenever $\b V^{\star}=0$. The corresponding
Lagrangian is
\[
L(\u,\b V,\xi,\mu)=\frac{1}{2}\|\u-\x\|^{2}+\frac{1}{2}\tr(\b V)+(\xi-\mu)^{T}(\b A\u-\b b)-\sum_{i=1}^{\ell}\left[\xi_{i}\sqrt{1+\langle\b c_{i}\b c_{i}^{T},\b V\rangle}+\mu_{i}\right],
\]
over Lagrange multipliers $\xi,\mu\ge0$. Assuming that $\b A^{T}\b A\ne0$,
this problem has strictly feasible primal point $\u=(\b A^{T}\b A)^{-1}\b b$
and $\b{V=\b I}$, and strictly feasible dual point $\xi=\mu=\epsilon\b 1$
for $\epsilon>0$. Hence, strong duality is attained as in
\[
\min_{\b V\succeq0,\u}\;\max_{\lambda,\mu\ge0}\;L(\u,\b V,\xi,\mu)=\max_{\lambda,\mu\ge0}\;\min_{\b V\succeq0,\u}\;L(\u,\b V,\xi,\mu).
\]
Examining the inner minimization over $\b V\succeq0$, note that the
associated optimiality conditions read
\[
\nabla_{\b V}L(\u,\b V^{\star},\xi,\mu)=\b S=\frac{1}{2}\left(I-\sum_{i=1}^{q}\frac{\xi_{i}\b c_{i}\b c_{i}^{T}}{\sqrt{1+\langle\b c_{i}\b c_{i}^{T},\b V^{\star}\rangle}}\right)\succeq0,\qquad\b S\b V^{\star}=0.
\]
Hence, the minimum is attained at $\b V^{\star}=0$ if and only if
$\sum_{i}\xi_{i}\b c_{i}\b c_{i}^{T}\prec I$. We will proceed to
solve the dual for the optimal Lagrange multiplier $\xi^{\star}$
and verify that $\sum_{i}\xi_{i}^{\star}\b c_{i}\b c_{i}^{T}\prec I$
is satisfied.

In the case that $\b V^{\star}=0,$ the corresponding $\u^{\star}$
is unique
\[
\u^{\star}=\arg\min_{\u}\;L(\u,0,\lambda,\mu)=\arg\min_{\u}\frac{1}{2}\|\u-\x\|^{2}+\b y^{T}(\b A\u-\b b)=\x-\b A^{T}\b y
\]
where $\b y=\xi-\mu$, and the dual problem is written
\begin{align*}
\max_{\xi,\mu\ge0}\;\min_{\u}\;L(\u,0,\xi,\mu) & =-\min_{\b y}\left\{ \frac{1}{2}\|\b A^{T}\b y\|^{2}-\b y^{T}(\b A\x-\b b)+\|\b y\|_{1}\right\} .
\end{align*}
whose optimal conditions read 
\[
\b A\b A^{T}\b y^{\star}-(\b A\x-\b b)\in\sign(\b y^{\star})\qquad\text{where }\sign(\alpha)=\begin{cases}
+1 & \alpha>0,\\{}
[-1,+1] & \alpha=0,\\
-1 & \alpha<0.
\end{cases}
\]
We wish to impose conditions on the data $\b A,\b b,\x$ to ensure
that $\lambda_{\max}(\max\{0,y_{i}^{\star}\}\b c_{i}\b c_{i}^{T})<1$
holds at dual optimality. A conservative condition is to use the enclosure
$\sign(\alpha)\subset[-1,+1]$ to solve a relaxation
\begin{align*}
 & \max_{\b y}\left\{ \lambda_{\max}\left(\sum_{i}\max\{0,\e_{i}^{T}\b y\}\b c_{i}\b c_{i}^{T}\right):\b y=(\b A\b A^{T})^{-1}(\b A\x-\b b-\b s),\quad\b s\in\sign(\b y)\right\} \\
\le & \lambda_{\max}\left(\sum_{i}\b c_{i}\b c_{i}^{T}\right)\cdot\max_{\b y}\left\{ \max_{i}\{0,\e_{i}^{T}\b y\}:\b y=(\b A\b A^{T})^{-1}(\b A\x-\b b-\b s),\quad\b s\in\sign(\b y)\right\} \\
\le & \lambda_{\max}\left(\sum_{i}\b c_{i}\b c_{i}^{T}\right)\cdot\max_{\b y}\left\{ \max_{i}\{0,\e_{i}^{T}\b y\}:\b y=(\b A\b A^{T})^{-1}(\b A\x-\b b-\b s),\quad\|\b s\|_{\infty}\le1\right\} \\
= & \lambda_{\max}\left(\sum_{i}\b c_{i}\b c_{i}^{T}\right)\cdot\max_{i}\left\{ 0,\e_{i}^{T}(\b A\b A^{T})^{-1}(\b A\x-\b b-\b s)\right\} .
\end{align*}
Hence, if the following holds
\[
\|\b C\|^{2}\cdot\max_{i}\left\{ \e_{i}^{T}(\b A\b A^{T})^{-1}(\b A\x-\b b)+\sum_{j=1}^{n}|\e_{i}^{T}(\b A\b A^{T})^{-1}\e_{j}|\right\} <1
\]
or more conservatively, if the following holds
\[
\|\b C\|^{2}\cdot(\|(\b A\b A^{T})^{-1}(\b A\x-\b b)\|_{\infty}+\|(\b A\b A^{T})^{-1}\|_{\infty})<1
\]
then $\b V^{\star}=0$ and $\u^{\star}$ is unique as desired. 
\end{proof}

\section{\label{sec:general}Further details for multiple layers}

In the general $\ell$-layer case, the SDP relaxation for problem
(\ref{eq:probA}) reads 
\begin{align}
d_{\lb}^{2}=\min_{\G_{k}\succeq0} & \tr(\X_{0})-2\langle\hat{\x},\x_{0}\rangle+\|\hat{\x}\|^{2}\tag{A-lb}\label{eq:relaxA3}\\
\text{s.t.}\quad & \begin{array}{c}
\x_{k+1}\ge0,\quad\x_{k+1}\ge\Wmat_{k}\x_{k}+\b b_{k},\\
\diag(\X_{k+1})\le\diag(\Wmat_{k}\b Y_{k}^{T}),\\
\langle\b w,\x_{\ell}\rangle+b\le0,
\end{array}\quad\G_{k}=\begin{bmatrix}1 & \x_{k}^{T} & \x_{k+1}^{T}\\
\x_{k} & \X_{k} & \b Y_{k}\\
\x_{k+1} & \b Y_{k}^{T} & \X_{k+1}
\end{bmatrix}\succeq0\;\forall k,\nonumber 
\end{align}
over layer indices $k\in\{0,1,\ldots,\ell-1\}$, while the SDP relaxation
for problem (\ref{eq:probB}) is almost identical, except the constraint
on $\x_{\ell}$:
\begin{align}
L_{\lb}^{2}=\min_{\G_{k}\succeq0} & \tr(\X_{0})-2\langle\hat{\x},\x_{0}\rangle+\|\hat{\x}\|^{2}\tag{B-lb}\label{eq:relaxB3}\\
\text{s.t.}\quad & \begin{array}{c}
\x_{k+1}\ge0,\quad\x_{k+1}\ge\Wmat_{k}\x_{k}+\b b_{k},\\
\diag(\X_{k+1})\le\diag(\Wmat_{k}\b Y_{k}^{T}),\\
\tr(\X_{\ell})-2\langle\hat{\z},\x_{\ell}\rangle+\|\hat{\z}\|^{2}\le\rho^{2},
\end{array}\quad\G_{k}=\begin{bmatrix}1 & \x_{k}^{T} & \x_{k+1}^{T}\\
\x_{k} & \X_{k} & \b Y_{k}\\
\x_{k+1} & \b Y_{k}^{T} & \X_{k+1}
\end{bmatrix}\succeq0\;\forall k,\nonumber 
\end{align}
over layer indices $k\in\{0,1,\ldots,\ell-1\}$. Note that both (\ref{eq:relaxA3})
and (\ref{eq:relaxB3}) are SDPs over $\ell$ smaller semidefinite
variables, each of order $1+n_{k}+n_{k+1}$, rather than over a single
large semidefinite variable of order $1+\sum_{k=1}^{\ell}n_{k}$.
This reduction is from an application of the chordal graph matrix
completion of \citet{fukuda2001exploiting}; see also~\citep[Chapter 10]{vandenberghe2015chordal}.

Now, for the choice $\hat{\z}=\u-\rho\b w/\|\b w\|$ where $\u=-b\b w/\|\b w\|^{2}$,
the optimal value $L^{\star}$ to problem (\ref{eq:probB}) gives
an upper-bound to the optimal value $L^{\star}\ge d^{\star}$ of problem
(\ref{eq:probA}) that converges to an equality at $\rho\to\infty$.
At the same time, $L_{\lb}\ge d_{\lb}$ holds for all $\rho>0$ because
problem (\ref{eq:relaxA3}) is always a relaxation of problem (\ref{eq:relaxB3}).
To show this, we observe that for this choice of $\hat{\z}$, we have
\begin{align*}
 & \tr(\X_{\ell})-2\langle\hat{\z},\x_{\ell}\rangle+\|\hat{\z}\|^{2}-\rho^{2}\\
= & \tr(\X_{\ell})-2\langle\u,\x_{\ell}\rangle+(2\rho/\|\b w\|)\langle\b w,\x_{\ell}\rangle+\|\b u\|^{2}-(2\rho/\|\b w\|)\langle\b w,\b u\rangle+\rho^{2}-\rho^{2}\\
= & \underbrace{\tr(\X_{\ell})-2\langle\u,\x_{\ell}\rangle+\|\u\|^{2}}_{\ge0}+(2\rho/\|\b w\|)[\langle\b w,\z\rangle+b].
\end{align*}
The nonnegativity of this first term follows because
\begin{align*}
\tr(\X_{\ell})-2\langle\u,\x_{\ell}\rangle+\|\u\|^{2} & =\tr(\X_{\ell}-\u\x_{\ell}^{T}-\x_{\ell}\u^{T}+\u\u^{T})\\
 & =\tr(\X_{\ell}-\x_{\ell}\x_{\ell}^{T}+(\x_{\ell}-\u)(\x_{\ell}-\u)^{T})\\
 & =\tr(\X_{\ell}-\x_{\ell}\x_{\ell}^{T})+\|\x_{\ell}-\u\|^{2}
\end{align*}
and that $\begin{bmatrix}1 & \x_{\ell}^{T}\\
\x_{\ell} & \X_{\ell}
\end{bmatrix}\succeq0\;$ implies $\X_{\ell}-\x_{\ell}\x_{\ell}^{T}\succeq0$ by the Schur
complement lemma and therefore $\tr(\X_{\ell}-\x_{\ell}\x_{\ell}^{T})\ge0$.
Hence, a feasible point $\X_{\ell},\x_{\ell}$ for the relaxation
(\ref{eq:relaxB3}) satisfying $\tr(\X_{\ell})-2\langle\hat{\z},\x_{\ell}\rangle+\|\hat{\z}\|^{2}\le\rho^{2}$
must immediately satisfy $\langle\b w,\z\rangle+b\le0$ and therefore
be feasible for the relaxation (\ref{eq:relaxA3}). 

If the relaxation (\ref{eq:relaxA3}) is tight, meaning that $d^{\star}=d_{\lb}$,
then the the relaxation (\ref{eq:relaxB3}) must automatically be
tight at $\rho\to\infty$, because $d^{\star}=L^{\star}\ge L_{\lb}\ge d_{\lb}=d^{\star}$.
But the converse need not hold: the relaxation (\ref{eq:relaxA3})
can still be loose even though (\ref{eq:relaxB3}) is tight, because
even with $L^{\star}=L_{\lb}$ at $\rho\to\infty$, it is still possible
to have $d_{\lb}<d^{\star}$.

The nonlinear interpretation of (\ref{eq:relaxB3}) reads
\begin{align}
\min_{\x_{i}^{(k)}\in\R^{p}}\quad & \sumsm_{j}\|\x_{0,j}-\hat{x}_{j}\e\|^{2}\quad\text{s.t. }\begin{array}{c}
\langle\e,\x_{i}^{(k+1)}\rangle\ge\max\left\{ 0,\langle\e,{\textstyle \sum_{j}}W_{i,j}^{(k)}\x_{j}^{(k)}+b_{i}^{(k)}\e\rangle\right\} ,\\
\|\x_{i}^{(k+1)}\|^{2}\le\langle\x_{i}^{(k+1)},{\textstyle \sum_{j}}W_{i,j}^{(k)}\x_{j}^{(k)}+b_{i}^{(k)}\e\rangle,\\
\sumsm_{j}\|\x_{\ell,j}-\hat{z}_{j}\e\|^{2}\le\rho^{2},
\end{array}\forall i,k\label{eq:ncvxB3}
\end{align}
over layer indices $k\in\{0,1,\ldots,\ell-1\}$ and neuron indices
$i\in\{1,2,\ldots,n\}$ at each $k$-th layer. Suppose that problem
(\ref{eq:probB}) has a trivial solution $\x^{\star}=\hat{\x}$ with
objective zero. Then, it follows that every solution to (\ref{eq:ncvxB3})
must be collinear and satisfy $\x_{0,j}^{\star}=\hat{x}_{j}\e$, so
the relaxation (\ref{eq:relaxB3}) has a unique rank-1 solution via
Theorem~\ref{thm:collinear}. 
\begin{proof}[Proof of Corollary~\ref{cor:multilayer}]
If $\|\b f(\hat{\x})-\hat{\z}\|\le\rho$, then problem (\ref{eq:ncvxB3})
has a minimum of zero, obtained by setting $\x_{0,j}^{\star}=\hat{x}_{j}\e$
for all $j$ at the input layer. This choice of $\x_{0,j}^{\star}$
is unique, because $\sumsm_{j}\|\x_{0,j}^{\star}-\hat{x}_{j}\e\|^{2}=0$
holds if and only if $\x_{0,j}^{\star}=\hat{x}_{j}\e$, so the input
layer must be collinear at optimality, meaning that $\|\x_{0,j}^{\star}\|=|\langle\e,\x_{0,j}^{\star}\rangle|$
for all $j$ is guaranteed to hold. Then, applying Lemma~\ref{lem:tight}
shows that $\|\x_{1,i}^{\star}-b_{i}^{(1)}\e\|=|\langle\e,\x_{1,i}^{\star}-b_{i}^{(1)}\e\rangle|$
and therefore $\|\x_{1,i}^{\star}\|=|\langle\e,\x_{1,i}^{\star}\rangle|$
for all $i$, so the first hidden layer is also collinear. Inductively
repeating this argument, if the $k$-th layer is collinear, as in
$\|\x_{k,j}^{\star}\|=|\langle\e,\x_{k,j}^{\star}\rangle|$ for all
$j$, then Lemma~\ref{lem:tight} shows that the $(k+1)$-th layer
is also collinear, as in $\|\x_{k+1,i}^{\star}\|=|\langle\e,\x_{k+1,i}^{\star}\rangle|$
for all $i$. Hence, all solutions to (\ref{eq:ncvxB3}) are collinear,
as in $\|\x_{k,j}^{\star}\|=|\langle\e,\x_{k,j}^{\star}\rangle|$
for all $j,k$. Evoking Theorem~\ref{thm:collinear} then yields
our desired claim.
\end{proof}

\section{\label{sec:BM2}The Rank-2 Burer--Monteiro Algorithm}

The Burer-Monteiro algorithm is obtained by using a local optimization
algorithm to solve the nonconvex interpretation of (\ref{eq:relaxB3})
stated in (\ref{eq:ncvxB3}). In particular, fix $p=2$, define at
the $k$-th layer $\u_{k}[j]=\langle\e,\x_{j}^{(k)}\rangle$ and $\v_{k}[j]=\sqrt{\|\x_{j}^{(k)}\|^{2}-\langle\e,\x_{j}^{(k)}\rangle^{2}}$
yields the rank-2 Burer--Monteiro problem (\ref{eq:bm2}) as desired.
In turn, given a solution $\{\u_{k}^{\star},\v_{k}^{\star}\}$ to
(\ref{eq:bm2}) satisfying $\v_{k}^{\star}=0$, we recover a \emph{rank
deficient} rank-2 solution to (\ref{eq:ncvxB3}) with $\x_{j}^{(k)}=\u_{k}^{\star}[j]\e$.
This rank-deficient solution is guaranteed to be globally optimal
if it satisfies first- and second-order optimality; see \citet{burer2005local}
and also~\citep{boumal2020deterministic,journee2010low} and in particular
\citep[Lemma~1]{cifuentes2019burer}.

Our MATLAB implementation solves problem (\ref{eq:bm2}) using \texttt{fmincon}
with \texttt{algorithm='interior-point'}, starting from an initial
point selected i.i.d. from the unit Gaussian, and terminating at relative
tolerances of $10^{-8}$. If the algorithm gets stuck at a spurious
local minimum with $\v_{0}^{\star}\ne0$, or if more than 300 iterations
or function calls are made, then we restart with a new random initial
point; we give up after 5 failed attempts. Empirically, we observed
that whenever the SDP relaxation is tight, \texttt{fmincon} would
consistently converge to a globally optimal solution satisfying $\v_{0}^{\star}\approx0$
within 80 iterations of the first attempt; this suggests an underlying
``no spurious local minima'' result like that of~\citet{boumal2020deterministic}. 
\end{document}